\documentclass{article}%
\usepackage{amsmath}
\usepackage{amsfonts}
\usepackage{amssymb}
\usepackage{graphicx}
\usepackage{hyperref}
\usepackage{subfigure}%
\providecommand{\U}[1]{\protect\rule{.1in}{.1in}}
\newtheorem{theorem}{Theorem}

\newtheorem{corollary}{Corollary}[theorem]

\newtheorem{preremark}{Remark}
\newtheorem{preexample}{Example}
\newenvironment{remark}{\begin{preremark}\rm}{\end{preremark}}
\newenvironment{example}{\begin{preexample}\rm}{\end{preexample}}
\newenvironment{proof}[1][Proof]{\noindent\textbf{#1.} }{$\square$ }
\begin{document}

\title{\textbf{ON THE DETERMINATION OF EXACT NUMBER OF LIMIT CYCLES IN LIENARD
SYSTEMS}}
\author{Aniruddha Palit$^{\ast}$ and Dhurjati Prasad Datta$^{\dagger}$\\$^{\ast}$Department of Mathematics, Surya Sen Mahavidyalaya,\\Siliguri, India, Pin - 734004. \\Email: mail2apalit@gmail.com\\$^{\dagger}$Department of Mathematics, University of North Bengal,\\Siliguri, India, Pin - 734013.\\Email: dp\_datta@yahoo.com}
\date{}
\maketitle

\begin{abstract}
We present a simpler proof of the existence of an exact number of one or more
limit cycles to the Lienard system $\dot{x}=y-F\left(  x\right)  $, $\dot
{y}=-g\left(  x\right)  $, under weaker conditions on the odd functions
$F\left(  x\right)  $ and $g\left(  x\right)  $ as compared to those available
in literature. We also give improved estimates of amplitudes of the limit
cycle of the Van Der Pol equation for various values of the nonlinearity
parameter. Moreover, the amplitude is shown to be independent of the
asymptotic nature of $F$ as $\left\vert x\right\vert \rightarrow\infty$.

\bigskip

\end{abstract}

\textit{Key words and ph\textit{r}ases:}\textbf{ }Autonomous system, Lienard
equation, Limit cycle.

\bigskip


\section{Introduction}

There has been a considerable interest in the study of the number and nature
of limit cycles in a Lienard equation%
\begin{equation}
\ddot{x}+f\left(  x\right)  \dot{x}+g\left(  x\right)  =0 \label{Lienard Eq}%
\end{equation}
recently $\left[  \text{\cite{Zheng Zuo-Huan}}-\text{\cite{Llibre Ponce
Torres}}\right]  $. Limit cycles are isolated periodic curves in the phase
plane and arise in numerous applications as self-sustained oscillations which
exist even in the absence of external periodic forcing. The equation $\left(
\ref{Lienard Eq}\right)  $ is usually studied as an autonomous system, called
the Lienard system, given by%
\begin{equation}
\dot{x}=y-F\left(  x\right)  ,\qquad\dot{y}=-g\left(  x\right)
\label{Lienard System}%
\end{equation}
where $F\left(  x\right)  =%
{\displaystyle\int\limits_{0}^{x}}
f\left(  u\right)  du$. The phase plane defined by $\left(
\ref{Lienard System}\right)  $ is called the Lienard plane. Lienard gave a
criterion for the uniqueness of periodic cycles for a general class of
equations when $F\left(  x\right)  $ is an odd function and satisfies a
monotonicity condition as $x\rightarrow\infty$. An interesting problem for the
system $\left(  \ref{Lienard System}\right)  $ is the determination of the
number of limit cycles for a given odd degree $\left(  m\right)  $ polynomial
$F\left(  x\right)  $. Lins, Pugh and de Melo $\cite{LPM}$ conjectured that
the system $\left(  \ref{Lienard System}\right)  $ has at most $N$ limit
cycles if $m=2N+1$ or $m=2N+2$. Currently this problem is being investigated
by many authors in connection with the still unsolved Hilbert's 16th problem.

Giacomini and Neukirch $\cite{Giacomini Neukirch2}$ have developed a general
procedure for constructing a sequence of polynomials whose roots of odd
multiplicity are related to the number and location of the limit cycles of
equation $\left(  \ref{Lienard Eq}\right)  $ when $f\left(  x\right)  $ is an
even degree polynomial. They have also given a sequence of algebraic
approximations to the equation of each such cycles, although their method is
mainly of experimental (numerical) in nature and a rigorous justification is
still lacking. Holst and Sundberg $\cite{Holst Sunburg}$ have extended
Rychkov's theorem $\cite{Rychkov}$ for a class of $F\left(  x\right)  $ having
$5th$ degree polynomial like behaviour. The proof of Richkov's theorem however
requires \textit{bifurcation theory}. Odani gave a proof on the existence of
exactly $N$ limit cycles of the Lienard equation $\left(  \ref{Lienard Eq}%
\right)  $ with $g\left(  x\right)  =x$. His proof does not make use of the
bifurcation theory. His method also gave an improved estimate of the amplitude
of a limit cycle. Recently there have been some progress in elucidating
sufficient conditions extending the previous results. Chen and Chen
$\cite{Chen Chen}$, for instance, proved the Lins-Pugh-de Melo conjecture for
Lienard system with function $F$ odd. On the other hand, it has been shown
$\cite{Dumortier Panazzolo Roussarie}$ that for suitable polynomial $F$ of
degree $7$, the system $\left(  \ref{Lienard System}\right)  $ has $4$ limit
cycles, contradicting the conjecture in $\cite{LPM}$. Chen, Llibre and Zhang
$\cite{Chen Llibre Zhang}$ proved a sufficient condition for existence of
exactly $N$ limit cycles for the system $\left(  \ref{Lienard System}\right)
$ with a general class of $F\left(  x\right)  $ functions. We investigate an
equivalent problem covering, however, a different class of functions $F$ as
compared to $\cite{Chen Llibre Zhang}$.

In the study on the number of limit cycles several authors have studied the
equation $\left(  \ref{Lienard Eq}\right)  $ in the usual phase plane $($viz.
Theorem $7.10-7.12$, Chapter $4$ in $\cite{Zhing Tongren Wenzao})$ whereas
some considered the Lienard plane $\cite{Chen Llibre Zhang}$. In Theorem
$7.10$, Chapter $4$ $\cite{Zhing Tongren Wenzao}$ the function $f$ is taken as
a periodic function. Theorem $7.11$ is a generalization of Theorem $7.10$ in
which the function $F^{\prime}\left(  x\right)  =f\left(  x\right)  $ is a
monotone function in certain regions. However Theorem $\ref{New Theorem}$ and
Theorem $\ref{New Theorem N}$ in the present paper do not depend upon the
monotonicity of $f$. Rather, we have used the monotonicity of $F$. As a
consequence, merely the sign of the function $f$ determines the monotonic
nature of $F$, and hence determines the number of limit cycles in Lienard
system $\left(  \ref{Lienard System}\right)  $. Thus our results cover a
different class of functions than those covered by Theorems $7.10$ and $7.11$
mentioned above. Theorem $7.12$, Chapter $4$ in $\cite{Zhing Tongren Wenzao}$
and the theorem in $\cite{Chen Llibre Zhang}$ have been proved on Lienard
plane. Both of these results have assumed the existence of $\beta_{j}%
\in\left[  a_{j},a_{j+1}\right]  $, $j=2,3,4,\ldots$ such that $F\left(
\beta_{j}\right)  =F\left(  L_{j-1}\right)  $ where, $a_{j}$'s are positive
roots of $F$ and $L_{j}$'s are unique extremum of $F$ in $\left[
a_{j},a_{j+1}\right]  $ for $j=1,2,3,\ldots$. However, if we do not get any
such $\beta_{j}$ then these results are not applicable. In such situations
Theorem $\ref{New Theorem N}$ in Section $\ref{New Theorem Section}$ of the
present paper is still applicable to determine the exact number of limit
cycles. One such example is given in section $\ref{Example}$.

In this paper we first give a simple but, nevertheless, an important extension
of the Lienard's theorem for the unique limit cycle by removing the unbounded
nature of the function $F$ as $x\rightarrow\infty$. Next, in Theorem
$\ref{New Theorem}$\ we prove that the system $\left(  \ref{Lienard System}%
\right)  $ has exactly two limit cycles when the odd function $F\left(
x\right)  $ undergoes two sign changes in $x>0$ and is monotonic not only as
$x\rightarrow\infty$, but also near $($actually at the right of$)$ the first
zero. However, $g\left(  x\right)  $ $(g\left(  x\right)  >0$ for $x>0)$ can
be any odd continuous function. Example $\ref{Ex Compare 2 Cycle}$ in support
of Theorem $\ref{New Theorem}$ reveals clearly the strength of this theorem
over analogous results $($e.g. Theorem $5.1)$ of $\cite{Zhing Tongren Wenzao}%
$. The \textit{new insights} gained from Theorem $\ref{New Theorem}$ $($and
also from Theorem $\ref{Lienard Theorem Extension})$ then provide a general
approach in obtaining an existence theorem for multiple limit cycles in a
systematic manner. In Theorem $\ref{New Theorem N}$, we state a set of such
conditions for the existence of exactly $N$ limit cycles. Although we are
dealing with odd functions $F$ only, there are certain odd functions as shown
in Example $\ref{Ex 3 Limit Cycle}$, which satisfy the conditions of Theorem
$\ref{New Theorem N}$ in the current paper but do not satisfy the theorem of
$\cite{Chen Llibre Zhang}$. This establishes our claim that the present
theorems cover different classes of functions $F$ than those covered in
$\cite{Zhing Tongren Wenzao}$ and $\cite{Chen Llibre Zhang}$. Moreover, as
stated above, $g\left(  x\right)  $ here is an odd function while for the
theorem of $\cite{Chen Llibre Zhang}$ $g\left(  x\right)  =x$. The second
important result that we find in section $\ref{Observations}$ is an efficient
upper estimate of the amplitude of the limit cycle for the system $\left(
\ref{Lienard System}\right)  $. The values of the amplitudes for the Van der
Pol equation are obtained in Example $\ref{Ex Compare Amplitude}$, which are
much more accurate compared to those in $\cite{Odani}$ and $\cite{Lopez}$.

The paper is organized as follows. In section $\ref{Lienard's Theorem Section}%
$, we sketch the main steps of the proof of the classical Lienard theorem thus
introducing our notations. In section $\ref{Observations}$, we discuss some
special observations leading to an extension of the classical Lienard theorem.
Our main result, Theorem $\ref{New Theorem}$, on the existence of two limit
cycles is proved in section $\ref{New Theorem Section}$. In Theorem
$\ref{New Theorem N}$ we state the sufficient conditions for existence of
exactly $N$ limit cycles. An outline of the proof is given in Appendix $($For
a detailed proof, see $\cite{Palit Datta 2})$. The proof of this general
existence theorem is based on an induction method with \textit{non-trivial
initial hypotheses} for $N=1$ and $N=2$. We present some examples in section
$\ref{Example}$ highlighting the key features of the above theorems. Section
$\ref{Current Status}$ contains some concluding remarks.

\section{Lienard's Theorem\label{Lienard's Theorem Section}}

Here we present an outline of the Lienard's Theorem for the sake of
completeness. This helps us introducing necessary notations which will be used subsequently.

\begin{theorem}
\label{Lienard Theorem}The equation $\left(  \ref{Lienard Eq}\right)  $ has a
unique periodic solution if\newline%
\begin{tabular}
[c]{rl}%
$\left(  i\right)  $ & $f$ and $g$ are continuous;\\
$\left(  ii\right)  $ & $F$ and $g\left(  x\right)  $ are odd functions with
$g\left(  x\right)  >0$ for $x>0$;\\
$\left(  iii\right)  $ & $F$ is zero only at $x=0$, $x=a$, $x=-a$ for some
$a>0;$\\
$\left(  iv\right)  $ & $F\left(  x\right)  \rightarrow\infty$ as
$x\rightarrow\infty$ monotonically for $x>a$.
\end{tabular}

\end{theorem}

\begin{proof}
[A Brief Sketch of the Proof]The general shape of the path can be obtained
from the following observations.

\begin{enumerate}
\item[$\left(  a\right)  $] Because of the symmetry of the system$\left(
\ref{Lienard Eq}\right)  $ under $\left(  x,y\right)  \rightarrow\left(
-x,-y\right)  $ any periodic orbit is symmetric about the origin.

\item[$\left(  b\right)  $] The slope of a phase path is given by%
\begin{equation}
\frac{dy}{dx}=\frac{-g\left(  x\right)  }{y-F\left(  x\right)  }\text{.}
\label{Gradient of Lienard System}%
\end{equation}
Thus, a phase path is horizontal if $\dfrac{dy}{dx}=0$, i.e. if $g\left(
x\right)  =0$, i.e. if $x=0$ $\left(  \text{by }\left(  ii\right)  \text{
above}\right)  $. Similarly, a phase path is vertical on the curve $y=F\left(
x\right)  $. Above the curve $y=F\left(  x\right)  $ we have $\dot{x}>0$ and
below $\dot{x}<0$. Moreover, $\dot{y}<0$ for $x>0$ and $\dot{y}>0$ for $x<0$.
\end{enumerate}

\begin{figure}[h]
\begin{center}
\includegraphics[height=6cm]{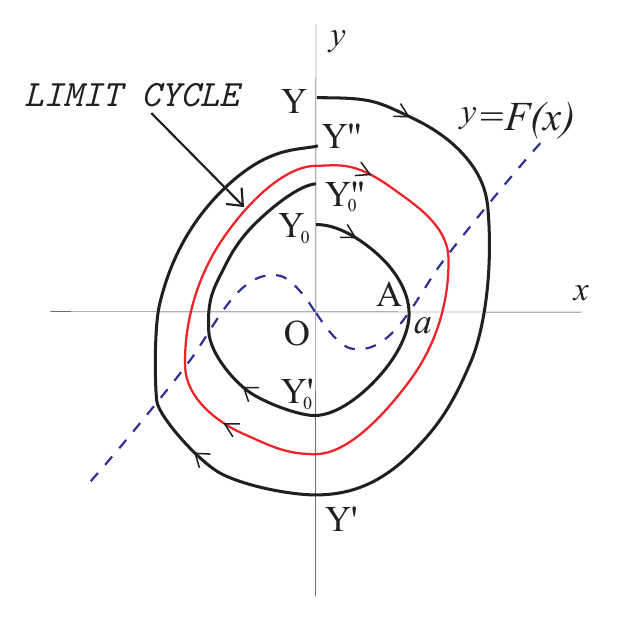}
\end{center}
\caption{{}Orbits of the Lienard System $\left(  \ref{Lienard System}\right)
$.}%
\label{Lienard Limit Cycle}%
\end{figure}

A path $YY^{\prime}Y^{\prime\prime}$ $($Figure \ref{Lienard Limit Cycle}$)$ is
closed iff $Y$ and $Y^{\prime\prime}$ coincide, which means by symmetry
$\left(  a\right)  $%
\begin{equation}
OY=OY^{\prime}\text{.} \label{Closed Path Condition}%
\end{equation}
This is equivalent to%
\begin{equation}
V_{YQY^{\prime}}=0 \label{Potential Zero}%
\end{equation}
where for a typical path $YQY^{\prime}$ in Figure \ref{Typical Path}%
\begin{equation}
V_{YQY^{\prime}}=v_{Y^{\prime}}-v_{Y}=%
{\displaystyle\int\limits_{YQY^{\prime}}}
dv=%
{\displaystyle\int\limits_{YQY^{\prime}}}
Fdy\text{.} \label{Potential Integral in v}%
\end{equation}
and%
\begin{equation}
v\left(  x,y\right)  =%
{\displaystyle\int\limits_{0}^{x}}
g\left(  u\right)  du+\frac{1}{2}y^{2}\text{.} \label{Potential Function}%
\end{equation}
\begin{figure}[h]
\begin{center}
\includegraphics[height=9cm]{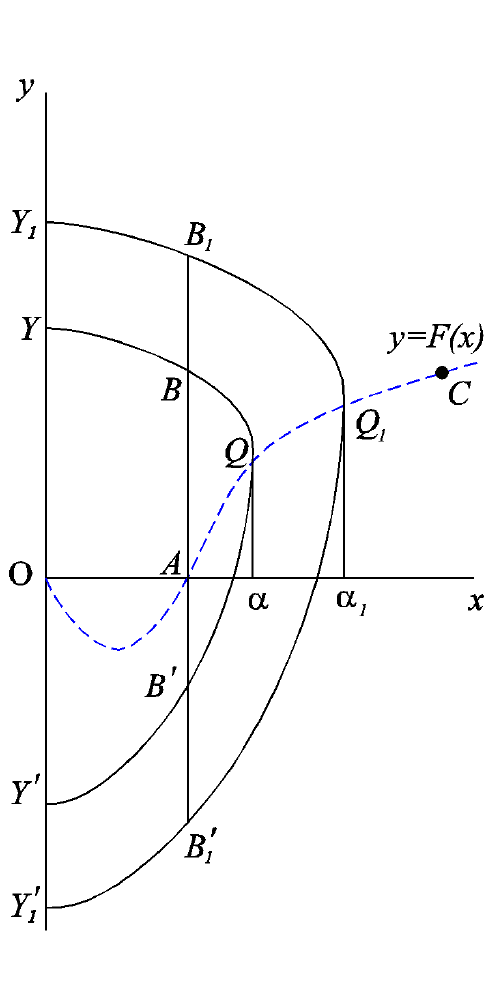}
\end{center}
\caption{{}Typical paths for the Lienard Theorem}%
\label{Typical Path}%
\end{figure}

Writing%
\[
V_{YQY^{\prime}}=V_{YB}+V_{BQB^{\prime}}+V_{B^{\prime}Y^{\prime}}%
\]
where $BB^{\prime}$ is a line parallel to the $y-$ axis and passing through
the point $\left(  0,a\right)  $ when the function $F$ changes its sign from
negative to positive, one then proves that

\begin{enumerate}
\item[$\left(  A\right)  $] As $Q$ moves out of the point $A\left(
0,a\right)  $ along the curve $AC$, the potentials $V_{YB}+V_{B^{\prime
}Y^{\prime}}$ is positive and monotone decreasing.

\item[$\left(  B\right)  $] As $Q$ moves out of the point $A\left(
0,a\right)  $ along the curve $AC$, $V_{BQB^{\prime}}$ is monotone decreasing.

\item[$\left(  C\right)  $] From $\left(  A\right)  $ and $\left(  B\right)  $
it follows that $V_{YQY^{\prime}}$ is monotone decreasing to the right of the
point $A$, $($Figure \ref{Typical Path}$)$.

\item[$\left(  D\right)  $] The quantity $V_{BQB^{\prime}}$ tends to $-\infty$
as the paths moves away to infinity.

\item[$\left(  E\right)  $] From $\left(  C\right)  $ and $\left(  D\right)
$, it follows that the quantity $V_{YQY^{\prime}}$ is monotone decreasing to
$-\infty$, at the right of the point $A$ $($Figure \ref{Typical Path}$)$.

\item[$\left(  F\right)  $] $V_{YQY^{\prime}}>0$ when the point $Q$ is at $A$
or to the left of the point $A$.
\end{enumerate}

It thus follows from $\left(  E\right)  $ and $\left(  F\right)  $ that
$V_{YQY^{\prime}}$ is monotone decreasing continuous function which changes
its sign from positive to negative as the point $Q$ moves out of $A\left(
a,0\right)  $ along the curve. As a result, $V_{YQY^{\prime}}$ will vanish
once and only once. Thus, there is one and only one closed path and the proof
is complete.
\end{proof}

\begin{remark}
\label{Simple Limit Cycle}The unique limit cycle in the above theorem is
\textbf{simple} in the sense that no (differentiable) perturbation satisfying
the conditions $\left(  i\right)  $-$\left(  iv\right)  $ can bifurcate the
limit cycle into two or more number of limit cycles.
\end{remark}

\begin{remark}
The condition $\left(  E\right)  $ enables us to conclude that once
$V_{YQY^{\prime}}$ becomes negative, it can never be positive as $Q$ moves to
infinity through the curve of $F\left(  x\right)  $. This observation helps us
to deduce the existence of a unique limit cycle. However, we see that the
existence of the limit cycle is indeed ensured only if $V_{YQY^{\prime}}$
becomes negative from positive i.e., if there is a change in sign of
$V_{YQY^{\prime}}$. Further, the unique value of $x$ for which $V_{YQY^{\prime
}}=0$ gives the \textbf{amplitude} of the limit cycle. Accordingly, if
$V_{YQY^{\prime}}$ becomes negative as $Q$ moves out from origin through the
curve of $F\left(  x\right)  $ then we get a limit cycle. \textit{This
observation actually gives one with a possibility of weakening the conditions
of the classical theorem, so as to accommodate a larger class of functions
}$F\left(  x\right)  $\textit{ but still having a unique limit cycle.} Theorem
$\ref{Lienard Theorem Extension}$ is one such realizations of a stronger
version of the classical theorem, which shows that the existence of the
$($unique$)$ limit cycle actually depends on the \textbf{local} monotonicity
of $F\left(  x\right)  $ on a bounded interval containing the point where
$V_{YQY^{\prime}}$ vanishes. A limit cycle can indeed be realized even when
$F\left(  x\right)  $ is bounded as $\left\vert x\right\vert \rightarrow
\infty$ $($c.f. Example $\ref{Ex Compare Amplitude})$\newline If it happens
further that $V_{YQY^{\prime}}$ becomes positive from negative once more, then
also, by an analogous argument as above we can get a point $Q$ on the curve
$F\left(  x\right)  $, through which another limit cycle must pass. To prove
this result we consider a function $F\left(  x\right)  $ $($in section
$\ref{New Theorem Section})$ which is monotonically increasing to the right of
the point $A$ for a sufficiently large value of $x$ and then it becomes
decreasing for some subsequent values of $x$ and ultimately become negative.
The proof depends on an efficient estimate of the amplitude of the first limit cycle.
\end{remark}

\section{Extension of the Classical Theorem and an Estimate of the
Amplitude\label{Observations}}

Let, $\left(  \alpha,F\left(  \alpha\right)  \right)  $ be the coordinate of
$Q$, as shown in Figure \ref{Typical Path} and let $\alpha=\hat{\alpha}$ be
the amplitude of the limit cycle of Theorem \ref{Lienard Theorem}. It is well
known that determining the exact value of limit cycle of the Lienard system is
a relatively difficult problem $\left[  \cite{Odani},\cite{Lopez}\right]  $.
We now find an estimate of $\hat{\alpha}$, for which the corresponding
$V_{YQY^{\prime}}$ just become negative from positive. Since, $V_{YQY^{\prime
}}$ is a monotone decreasing continuous function as the point moves out of the
point $A\left(  a,0\right)  $ along the curve, without any loss of generality
we can say $V_{YQY^{\prime}}$ can just become negative from positive if at
least one of the following two cases hold, viz.,

$\left(  i\right)  $ $\ V_{YQ}=0$ but $V_{QY^{\prime}}<0$

$\left(  ii\right)  $ $V_{QY^{\prime}}=0$ but $V_{YQ}<0$.\newline The third
possibility $V_{YQ}<0$ and $V_{QY^{\prime}}<0$ can be reduced to either of the
above two cases by monotonicity and continuity of $V_{YQY^{\prime}}$, i.e. by
taking an $\alpha$ closer to $\hat{\alpha}$, $($i.e., $\alpha\rightarrow
\hat{\alpha}+0$ $)$ either one of $V_{YQ}$ and $V_{QY^{\prime}}$ can be made
to vanish. Similarly, the possibility that either one of $V_{YQ}$ and
$V_{QY^{\prime}}$ is positive while their sum is negative, can also be
eliminated by choosing $\alpha$ far from $\hat{\alpha}$ $\left(  \alpha
>\hat{\alpha}\right)  $.

\subsubsection*{Case $\left(  i\right)  $}

Here $V_{YQ}=0$ is possible if%
\begin{equation}
V_{YB}+V_{BQ}=0 \label{Potential Zero Case I}%
\end{equation}
In step $\left(  A\right)  $ of the proof of Theorem $\ref{Lienard Theorem}$
it has been proved that $V_{YB}>0$.

We are now going to show that $V_{BQ}<0$.

On the path $BQ$, we have $F\left(  x\right)  \geq0$ and $\dfrac{dy}{dt}%
=\dot{y}=-g\left(  x\right)  <0$.

Therefore,
\begin{align}
V_{BQ}  &  =%
{\displaystyle\int\limits_{BQ}}
Fdy=%
{\displaystyle\int\limits_{BQ}}
F\frac{dy}{dt}dt\nonumber\\
&  =-%
{\displaystyle\int\limits_{BQ}}
F\left(  x\left(  t\right)  \right)  g\left(  x\left(  t\right)  \right)
dt\leq0\text{.} \label{Potential Integral in t}%
\end{align}
Thus, we can say that $\left(  \text{\ref{Potential Zero Case I}}\right)  $ is
true if $\left\vert V_{YB}\right\vert =\left\vert V_{BQ}\right\vert $%
\[
\text{i.e., if }\left\vert v\left(  a,y_{+}\left(  a\right)  \right)
-v\left(  0,y_{+}\left(  0\right)  \right)  \right\vert =\left\vert v\left(
\alpha,F\left(  \alpha\right)  \right)  -v\left(  a,y_{+}\left(  a\right)
\right)  \right\vert
\]
\newline where $y_{+}\left(  0\right)  =OY$ $($Figure \ref{Typical Path}$)$.
It is possible if%
\[%
{\displaystyle\int\limits_{0}^{a}}
g\left(  u\right)  du+\dfrac{1}{2}y_{+}^{2}\left(  a\right)  -\dfrac{1}%
{2}y_{+}^{2}\left(  0\right)  =-%
{\displaystyle\int\limits_{a}^{\alpha}}
g\left(  u\right)  du-\dfrac{1}{2}F^{2}\left(  \alpha\right)  +\dfrac{1}%
{2}y_{+}^{2}\left(  a\right)
\]
So, we have%
\begin{equation}
G\left(  \alpha\right)  =\dfrac{1}{2}y_{+}^{2}\left(  0\right)  -\dfrac{1}%
{2}F^{2}\left(  \alpha\right)  \label{Definition of alfa+}%
\end{equation}
where $G\left(  x\right)  =%
{\displaystyle\int\limits_{0}^{x}}
g\left(  u\right)  du$.

Let $\alpha=\alpha^{\prime}$ be a root of $\left(
\text{\ref{Definition of alfa+}}\right)  $ $\left(  \text{existence of which
is assured by construction}\right)  $ so that%
\begin{equation}
G\left(  \alpha^{\prime}\right)  =\dfrac{1}{2}y_{+}^{2}\left(  0\right)
-\dfrac{1}{2}F^{2}\left(  \alpha^{\prime}\right)  \text{.}
\label{Definition of alfaDash1}%
\end{equation}

\subsubsection*{Case $\left(  ii\right)  $}

Here, $V_{QY^{\prime}}=0$ is possible if%
\begin{equation}
V_{QB^{\prime}}+V_{B^{\prime}Y^{\prime}}=0 \label{Potential Zero Case II}%
\end{equation}
In step $\left(  A\right)  $ of the proof of the Theorem
$\ref{Lienard Theorem}$ it is proved that $V_{B^{\prime}Y^{\prime}}>0$.

Proceeding analogous to case $\left(  i\right)  $ one establishes that
$V_{QB^{\prime}}<0$ and consequently $\left(  \ref{Potential Zero Case II}%
\right)  $ is true provided%
\begin{equation}
G\left(  \alpha\right)  =\dfrac{1}{2}y_{-}^{2}\left(  0\right)  -\dfrac{1}%
{2}F^{2}\left(  \alpha\right)  \label{Definition of alfa-}%
\end{equation}
where $OY^{\prime}=-y_{-}\left(  0\right)  ,~y_{-}\left(  0\right)  <0$
$($Figure \ref{Typical Path}$)$. If $\alpha=\alpha^{\prime\prime}$ be a root
of $\left(  \text{\ref{Definition of alfa-}}\right)  $ we have%
\begin{equation}
G\left(  \alpha^{\prime\prime}\right)  =\dfrac{1}{2}y_{-}^{2}\left(  0\right)
-\dfrac{1}{2}F^{2}\left(  \alpha^{\prime\prime}\right)  \text{.}
\label{Definition of alfaDash2}%
\end{equation}

It now follows that if we take%
\begin{equation}
\bar{\alpha}=\max\left\{  \alpha^{\prime},\alpha^{\prime\prime}\right\}
\label{Definition of alfaBar}%
\end{equation}
then for any value of $\alpha>\bar{\alpha},$ $V_{YQY^{\prime}}\leq0\ $since,
$V_{YQY^{\prime}}=V_{YQ}+V_{QY^{\prime}}$. Thus the function $F$ should be
monotonic increasing in the interval $a<x\leq\bar{\alpha}$. Notice that the
classical Lienard theorem already ensures the existence of such an
$\bar{\alpha}$. In the light of the above discussion we can now
\textit{extend} the classical Lienard theorem by weakening the unbounded
nature of the function $F$ as stated in the following theorem and cover a more
large class of functions.

\begin{theorem}
\label{Lienard Theorem Extension}The equation $\left(  \ref{Lienard Eq}%
\right)  $ has a unique limit cycle if\newline%
\begin{tabular}
[c]{rl}%
$\left(  i\right)  $ & $f$ and $g$ are continuous in $\left(  -d,d\right)  $
for sufficiently large $d$;\\
$\left(  ii\right)  $ & $F$ and $g$ are odd functions with $g\left(  x\right)
>0$ for $x>0$;\\
$\left(  iii\right)  $ & $F$ is zero only at $x=0$, $x=a$, $x=-a$ for some $a$
where $0<a<d;$\\
$\left(  iv\right)  $ & $\exists$ a number $\bar{\alpha}$ defined by $\left(
\ref{Definition of alfaBar}\right)  $ such that $F$ is monotonic increasing
in\\
& $a<x\leq\bar{\alpha}$ and nondecreasing in $\bar{\alpha}<x<d$.
\end{tabular}

\end{theorem}

The existence of $\bar{\alpha}$ ensures a sign change in $V_{YQY^{\prime}}$
whereby we get the existence of a unique limit cycle in the finite phase
plane. The remaining part of the proof of this theorem remain same as that of
classical Lienard theorem. Also, from the above discussion and the proof of
classical Lienard theorem it follows that $V_{YQY^{\prime}}$ does not change
its sign any more if the function $F$ is simply \textit{monotone
nondecreasing} in $\bar{\alpha}<x<\infty$. In such cases the function $F$ can
\textit{even be bounded and even attain a constant value as }$x\rightarrow
\infty$, but still we get a unique limit cycle for such a bounded Lienard
system. Thus, we can indeed cover a larger class of functions than those
covered by the Lienard theorem $($c.f. \ref{Ex Compare Amplitude}$)$.

In the beginning of the proof of Theorem \ref{Lienard Theorem} we observed
that above the curve $y=F\left(  x\right)  $ we have $\dot{x}>0$ and below
$\dot{x}<0$. So, the $x$-coordinate of a point on a limit cycle will achieve
its maximum absolute value on the curve $y=F\left(  x\right)  $. Therefore,
the amplitude of a limit cycle for the Lienard system is the abscissa of the
point $Q$ lying on the curve $y=F\left(  x\right)  $. Since, for a limit cycle
we have $V_{YQY^{\prime}}=0$, so by the construction of $\bar{\alpha}$ it
follows that it is an \textit{efficient upper estimate of the amplitude} of
the limit cycle. In the following example we find the values of $\bar{\alpha}$
for the well known Van der Pol equation against different values of $\mu$ and
compare them with the results obtained in $\cite{Odani}$ and $\cite{Lopez}$.
This also gives an example of a \textit{bounded} Van der Pol equation having
same amplitude as that of the standard Van der Pol equation.

\begin{example}
\label{Ex Compare Amplitude}Here, in the following table we present estimates
of the amplitude of the limit cycle for Van der Pol equation in which
$F\left(  x\right)  =\mu\left(  \dfrac{x^{3}}{3}-x\right)  $ and $g\left(
x\right)  =x$ for different values of $\mu.$ It is clear that our estimates
are reasonably close to the exact $($numerically computed$)$ values $($as
reported in $\cite{Odani})$. Our values also appear to be much better than the
upper bound $2.3233$ of $\cite{Odani}$ $($the estimated values of
$\cite{Lopez}$ are valid only for small $\mu)$. From our numerical estimates
it follows that although the estimated values of the amplitude seem to vary
irregularly for the moderately large values of $\mu\in\left[  0,10\right]  $,
these are nevertheless bounded above by $2.05$.%
\[%
\begin{tabular}
[c]{|c|c|c|}\hline
$\mu$ & $y_{+}\left(  0\right)  $ and $y_{-}\left(  0\right)  $ & $\bar
{\alpha}$\\\hline
$0.1$ & $2.00117$ & $2.0000586437166383$\\\hline
$0.2$ & $2.007076$ & $2.002540101136999$\\\hline
$0.3$ & $2.015912$ & $2.0054678254782505$\\\hline
$0.4$ & $2.028253$ & $2.0091503375996034$\\\hline
$0.5$ & $2.044065$ & $2.013278539452526$\\\hline
$1$ & $2.1727135$ & $2.0327736318429275$\\\hline
$1.5$ & $2.3710897$ & $2.0436704679281523$\\\hline
$2$ & $2.6149725$ & $2.04739132291152$\\\hline
$2.5$ & $2.8844602$ & $2.047213463900291$\\\hline
$3$ & $3.1687156$ & $2.045311842105752$\\\hline
$3.5$ & $3.462322$ & $2.0427848260891426$\\\hline
$4.5$ & $4.06701715$ & $2.037557405718347$\\\hline
$5$ & $4.3752293$ & $2.035154629371522$\\\hline
$10$ & $7.5528123$ & $2.020095969119061$\\\hline
\end{tabular}
\ \ \ \ \
\]
We get the same result if we consider the bounded function%
\[
F\left(  x\right)  =\left\{
\begin{array}
[c]{ll}%
\mu\left(  \dfrac{x^{3}}{3}-x\right)  & x\in\left(  -2.4,2.4\right) \\
\mu\left(  \dfrac{(2.4)^{3}}{3}-2.4\right)  -\dfrac{4.76\mu}{\sin\left(
0.6\right)  }\left(  \cos\left(  0.6\right)  -\cos\left(  x-3\right)  \right)
& x\in\left(  -3,-2.4\right]  \cup\left[  2.4,3\right) \\
\mu\left(  \dfrac{(2.4)^{3}}{3}-2.4\right)  -\dfrac{4.76\mu}{\sin\left(
0.6\right)  }\left(  \cos\left(  0.6\right)  -1\right)  & x\in\left(
-\infty,-3\right]  \cup\left[  3,\infty\right)
\end{array}
\right.
\]
in the whole phase plane or the following function in finite phase plane.%
\[
F\left(  x\right)  =\left\{
\begin{array}
[c]{ll}%
\mu\left(  \dfrac{x^{3}}{3}-x\right)  & x\in\left(  -2.4,2.4\right) \\
\mu\left(  \dfrac{(2.4)^{3}}{3}-2.4\right)  -\dfrac{4.76\mu}{\sin\left(
0.6\right)  }\left(  \cos\left(  0.6\right)  -\cos\left(  x-3\right)  \right)
& x\in\left(  -3,-2.4\right]  \cup\left[  2.4,3\right)
\end{array}
\right.
\]
This also tells that the \textit{value of }$d$\textit{ need not by too large}.
A limit cycle is assured even for a moderately large $d$.
\end{example}

It follows from this example that the amplitude of the unique limit cycle is
independent of the asymptotic behaviour of $F\left(  x\right)  $ as
$x\rightarrow\infty$. Indeed, the amplitude corresponds to the point $Q\left(
\hat{\alpha},F\left(  \hat{\alpha}\right)  \right)  $ on the limit cycle for
which the integral $V_{YQY^{\prime}}=%
{\textstyle\int\limits_{YQY^{\prime}}}
F~dy$ vanishes and clearly depends on the form of $F\left(  x\right)  $ in the
finite interval $\left(  0,\hat{\alpha}\right)  $. To the best of author's
knowledge, this result apparently is not recorded clearly in the literature.
We therefore state this observation as the following corollary.

\begin{corollary}
Amplitude of the unique limit cycle of the Lienard system $\left(
\ref{Lienard System}\right)  $ is independent of the asymptotic behaviour of
$F\left(  x\right)  $ as $\left\vert x\right\vert \rightarrow\infty$.
\end{corollary}

\section{The New Theorem\label{New Theorem Section}}

By the observations discussed in last section it is clear that, in the
interval $a<\alpha<\bar{\alpha}$ we obtain the limit cycle of Theorems
$\ref{Lienard Theorem}$ and $\ref{Lienard Theorem Extension}$. Moreover,
because of the condition $\left(  iv\right)  $ we see that the limit cycle
remains unique. However, if the function does not satisfy this condition, then
the limit cycle may not be unique. We now present our new theorem.

\begin{theorem}
\label{New Theorem}Let $f$ and $g$ be two functions satisfying the following
properties.\newline%
\begin{tabular}
[c]{rl}%
$\left(  i\right)  $ & $f$ and $g$ are continuous;\\
$\left(  ii\right)  $ & $F$ and $g$ are odd functions and $g\left(  x\right)
>0$ for $x>0$.;\\
$\left(  iii\right)  $ & $F$ has positive simple zeros only at $x=a_{1}$,
$x=a_{2}$ for some $a_{1}>0$ and\\
& some $a_{2}>\bar{\alpha}$, $\bar{\alpha}$ being defined by $\left(
\ref{Definition of alfaBar}\right)  $ and $\bar{\alpha}<L$, where $L$ is the
first\\
& local maxima of $F\left(  x\right)  $ in $\left[  a_{1},a_{2}\right]  ;$\\
$\left(  iv\right)  $ & $F$ is monotonic increasing in $a_{1}<x\leq\bar
{\alpha}$ and $F\left(  x\right)  \rightarrow-\infty$ as $x\rightarrow\infty
$\\
& monotonically for $x>a_{2};$%
\end{tabular}
\newline Then the equation $\left(  \ref{Lienard Eq}\right)  $ has exactly two
limit cycles around the origin.
\end{theorem}

\textbf{Proof:} We can get exactly the same observations as we get in
observations $\left(  a\right)  $ and $\left(  b\right)  $ in beginning of the
proof of Theorem $\ref{Lienard Theorem}$.

\begin{figure}[h]
\begin{center}
\includegraphics[height=9cm]{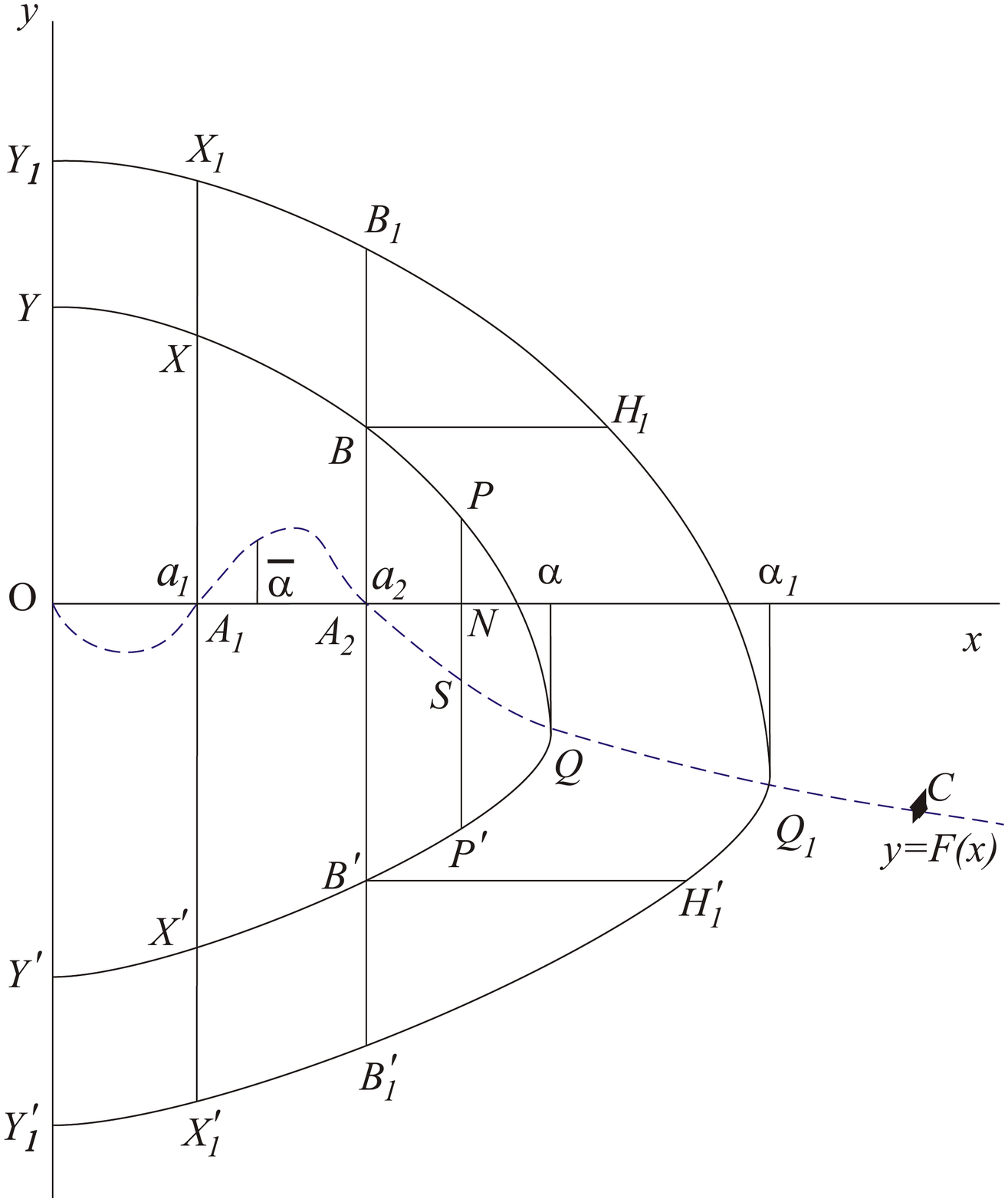}
\end{center}
\caption{{}}%
\label{New Typical Path}%
\end{figure}

By the observations in section $\ref{Observations}$ we can ensure the
existence of inner limit cycle. So, we shall now prove the existence of one
more limit cycle by showing that%
\[
OY=OY^{\prime}%
\]
once more when $x>\bar{\alpha}$. To prove the result we shall consider the
function $v\left(  x,y\right)  $ as in $\left(  \ref{Potential Function}%
\right)  $, and write,%
\begin{equation}
V_{YQY^{\prime}}=V_{YX}+V_{XB}+V_{BQB^{\prime}}+V_{B^{\prime}X^{\prime}%
}+V_{X^{\prime}Y^{\prime}} \label{Potential Decomposition}%
\end{equation}
where, $XX^{\prime}$ is a line parallel to the $y-$ axis passing through the
point $\left(  0,a_{1}\right)  $ where the function $F$ changes its sign from
negative to positive and $BB^{\prime}$ is a line parallel to the $y-$ axis
passing through the point $\left(  0,a_{2}\right)  $ where the function $F$
changes its sign from positive to negative. The proof is carried out through
the steps $\left(  A\right)  $ to $\left(  F\right)  $ below. Here we refer to
the Figure $\ref{New Typical Path}$.

\subsubsection*{Step $\left(  A\right)  :$ As $Q$ moves out from $A_{2}$ along
$A_{2}C$, $V_{YX}+V_{X^{\prime}Y^{\prime}}$ is positive and monotonic
decreasing.}

We choose two points $Q\left(  \alpha,F\left(  \alpha\right)  \right)  $ and
$Q_{1}\left(  \alpha_{1},F\left(  \alpha_{1}\right)  \right)  $ on the curve
of $F\left(  x\right)  $ where $\alpha_{1}>\alpha$. Let $YQY^{\prime}$ and
$Y_{1}Q_{1}Y_{1}^{\prime}$ be two paths through $Q$ and $Q_{1}$ respectively.
On the segments $YX$ and $Y_{1}X_{1}$ we have%
\[
y>0,~F\left(  x\right)  <0\text{ and }y-F\left(  x\right)  >0\text{.}%
\]
Now,%
\begin{equation}%
\begin{tabular}
[c]{lc}
& $\left(  y-F\left(  x\right)  \right)  _{YX}<\left(  y-F\left(  x\right)
\right)  _{Y_{1}X_{1}}$\\
$\implies$ & $\left(  \dfrac{1}{y-F\left(  x\right)  }\right)  _{YX}>\left(
\dfrac{1}{y-F\left(  x\right)  }\right)  _{Y_{1}X_{1}}$%
\end{tabular}
\nonumber
\end{equation}
Since $g\left(  x\right)  >0$ for $x>0$, we have%
\[
\left(  \dfrac{-g\left(  x\right)  }{y-F\left(  x\right)  }\right)
_{YX}<\left(  \dfrac{-g\left(  x\right)  }{y-F\left(  x\right)  }\right)
_{Y_{1}X_{1}}%
\]
So by $\left(  \ref{Gradient of Lienard System}\right)  $ we get%
\begin{equation}
\left(  \dfrac{dy}{dx}\right)  _{YX}<\left(  \dfrac{dy}{dx}\right)
_{Y_{1}X_{1}}<0 \label{Gradient Comparison1}%
\end{equation}
Therefore,%
\begin{equation}
V_{YX}=%
{\displaystyle\int\limits_{YX}}
Fdy=%
{\displaystyle\int\limits_{YX}}
\left(  -F\right)  \left(  -\frac{dy}{dx}\right)  dx\nonumber
\end{equation}
Using $\left(  \ref{Gradient Comparison1}\right)  $ we get%
\[
V_{YX}>%
{\displaystyle\int\limits_{Y_{1}X_{1}}}
\left(  -F\right)  \left(  -\frac{dy}{dx}\right)  dx
\]
Since, $F$ and $dy$ are positive on $Y_{1}X_{1}$ we have%
\begin{equation}
V_{YX}>%
{\displaystyle\int\limits_{Y_{1}X_{1}}}
Fdy=V_{Y_{1}X_{1}}>0\text{.} \label{Potential Comparison1}%
\end{equation}
Next, on the segments $X^{\prime}Y^{\prime}$ and $X_{1}^{\prime}Y_{1}^{\prime
}$ we have%
\[
y<0,~F\left(  x\right)  <0\text{ and }y-F\left(  x\right)  <0\text{.}%
\]
Now,%
\[%
\begin{tabular}
[c]{lcl}
& $\left(  y-F\left(  x\right)  \right)  _{X^{\prime}Y^{\prime}}>\left(
y-F\left(  x\right)  \right)  _{X_{1}^{\prime}Y_{1}^{\prime}}$ & \\
$\implies$ & $\left(  \dfrac{-g\left(  x\right)  }{y-F\left(  x\right)
}\right)  _{X^{\prime}Y^{\prime}}>\left(  \dfrac{-g\left(  x\right)
}{y-F\left(  x\right)  }\right)  _{X_{1}^{\prime}Y_{1}^{\prime}}$ &
\end{tabular}
\
\]
So, by $\left(  \ref{Gradient of Lienard System}\right)  $%
\begin{equation}
\left(  \dfrac{dy}{dx}\right)  _{X^{\prime}Y^{\prime}}>\left(  \dfrac{dy}%
{dx}\right)  _{X_{1}^{\prime}Y_{1}^{\prime}}%
>0.\ \ \ \label{Gradient Comparison2}%
\end{equation}
Therefore,%
\begin{equation}
V_{X^{\prime}Y^{\prime}}=%
{\displaystyle\int\limits_{X^{\prime}Y^{\prime}}}
Fdy=%
{\displaystyle\int\limits_{Y^{\prime}X^{\prime}}}
\left(  -F\right)  \frac{dy}{dx}dx\nonumber
\end{equation}
Using $\left(  \ref{Gradient Comparison2}\right)  $ we get%
\[
V_{X^{\prime}Y^{\prime}}>%
{\displaystyle\int\limits_{Y_{1}^{\prime}X_{1}^{\prime}}}
\left(  -F\right)  \frac{dy}{dx}dx~
\]
Since, $F$ and $dy$ are negative on $X_{1}^{\prime}Y_{1}^{\prime}$ we have%
\begin{equation}
V_{X^{\prime}Y^{\prime}}>%
{\displaystyle\int\limits_{X_{1}^{\prime}Y_{1}^{\prime}}}
Fdy=V_{X_{1}^{\prime}Y_{1}^{\prime}}>0\text{.} \label{Potential Comparison2}%
\end{equation}
From $\left(  \ref{Potential Comparison1}\right)  $ and $\left(
\ref{Potential Comparison2}\right)  $ we have%
\[
V_{YX}+V_{X^{\prime}Y^{\prime}}>V_{Y_{1}X_{1}}+V_{X_{1}^{\prime}Y_{1}^{\prime
}}>0\text{.}%
\]
Therefore, $V_{YX}+V_{X^{\prime}Y^{\prime}}$ is positive and monotonic
decreasing as the point $Q$ moves out from $A_{2}$ along $A_{2}C$.

\subsubsection*{Step $\left(  B\right)  :$ As $Q$ moves out from $A_{2}$ along
$A_{2}C$, $V_{XB}+V_{B^{\prime}X^{\prime}}$ is negative and monotonic
increasing.}

On the segments $XB$ and $X_{1}B_{1}$ we have%
\[
y>0,~F\left(  x\right)  <0\text{ and }y-F\left(  x\right)  >0\text{.}%
\]
Now,%
\[%
\begin{tabular}
[c]{lcl}
& $\left(  y-F\left(  x\right)  \right)  _{XB}<\left(  y-F\left(  x\right)
\right)  _{X_{1}B_{1}}$ & \\
$\implies$ & $\left(  \dfrac{-g\left(  x\right)  }{y-F\left(  x\right)
}\right)  _{XB}<\left(  \dfrac{-g\left(  x\right)  }{y-F\left(  x\right)
}\right)  _{X_{1}B_{1}}$ &
\end{tabular}
\]
So, by $\left(  \ref{Gradient of Lienard System}\right)  $ we get%
\begin{equation}
\left(  \dfrac{dy}{dx}\right)  _{XB}<\left(  \dfrac{dy}{dx}\right)
_{X_{1}B_{1}}<0\ \ \ \label{Gradient Comparison3}%
\end{equation}
Therefore,%
\begin{equation}
V_{XB}=%
{\displaystyle\int\limits_{XB}}
Fdy=%
{\displaystyle\int\limits_{XB}}
F\frac{dy}{dx}dx\nonumber
\end{equation}
Using $\left(  \ref{Gradient Comparison3}\right)  $ we get%
\begin{equation}
V_{XB}<%
{\displaystyle\int\limits_{X_{1}B_{1}}}
F\frac{dy}{dx}dx=%
{\displaystyle\int\limits_{X_{1}B_{1}}}
Fdy=V_{X_{1}B_{1}}<0 \label{Potential Comparison3}%
\end{equation}
\newline since, $F>0$ and $dy<0$ on $XB$ and $X_{1}B_{1}$.

Next, on the segments $B^{\prime}X^{\prime}$ and $B_{1}^{\prime}X_{1}^{\prime
}$ we have%
\[
y<0,~F\left(  x\right)  >0\text{ and }y-F\left(  x\right)  <0\text{.}%
\]
Now,%
\[%
\begin{tabular}
[c]{lc}
& $\left(  y-F\left(  x\right)  \right)  _{B^{\prime}X^{\prime}}>\left(
y-F\left(  x\right)  \right)  _{B_{1}^{\prime}X_{1}^{\prime}}$\\
$\implies$ & $\left(  \dfrac{-g\left(  x\right)  }{y-F\left(  x\right)
}\right)  _{B^{\prime}X^{\prime}}>\left(  \dfrac{-g\left(  x\right)
}{y-F\left(  x\right)  }\right)  _{B_{1}^{\prime}X_{1}^{\prime}}$%
\end{tabular}
\]
Using $\left(  \ref{Gradient of Lienard System}\right)  $ we get%
\begin{equation}
\left(  \dfrac{dy}{dx}\right)  _{B^{\prime}X^{\prime}}>\left(  \dfrac{dy}%
{dx}\right)  _{B_{1}^{\prime}X_{1}^{\prime}}%
>0.\ \ \label{Gradient Comparison4}%
\end{equation}
Therefore,%
\begin{equation}
V_{B^{\prime}X^{\prime}}=%
{\displaystyle\int\limits_{B^{\prime}X^{\prime}}}
Fdy=-%
{\displaystyle\int\limits_{X^{\prime}B^{\prime}}}
F\frac{dy}{dx}dx=%
{\displaystyle\int\limits_{X^{\prime}B^{\prime}}}
F\left(  -\frac{dy}{dx}\right)  dx\nonumber
\end{equation}
So, by $\left(  \ref{Gradient Comparison4}\right)  $ we have%
\begin{equation}
V_{B^{\prime}X^{\prime}}<%
{\displaystyle\int\limits_{X_{1}^{\prime}B_{1}^{\prime}}}
F\left(  -\frac{dy}{dx}\right)  dx~=%
{\displaystyle\int\limits_{B_{1}^{\prime}X_{1}^{\prime}}}
Fdy=V_{B_{1}^{\prime}X_{1}^{\prime}}<0\text{.} \label{Potential Comparison4}%
\end{equation}
\newline since, $F>0$ and $dy<0$ on $B^{\prime}X^{\prime}$ and $B_{1}^{\prime
}X_{1}^{\prime}$.

From $\left(  \ref{Potential Comparison3}\right)  $ and $\left(
\ref{Potential Comparison4}\right)  $ we have%
\[
V_{XB}+V_{B^{\prime}X^{\prime}}<V_{X_{1}B_{1}}+V_{B_{1}^{\prime}X_{1}^{\prime
}}<0\text{.}%
\]
Therefore, $V_{XB}+V_{B^{\prime}X^{\prime}}$ is negative and monotonic
increasing as the point $Q$ moves out from $A_{2}$ along $A_{2}C$.

\subsubsection*{Step $\left(  C\right)  :$ As $Q$ moves out from $A_{2}$ along
$A_{2}C$, $V_{BQB^{\prime}}$ is positive and monotonic increasing and tends to
$+\infty$ as the path recedes to infinity.}

On $BQB^{\prime}$ and $B_{1}Q_{1}B_{1}^{\prime}$, we have $F\left(  x\right)
<0$. We draw $BH_{1}$ and $B^{\prime}H_{1}^{\prime}$ parallel to $x-$ axis.

Therefore,%
\begin{align*}
V_{B_{1}Q_{1}B_{1}^{\prime}}  &  =%
{\displaystyle\int\limits_{B_{1}Q_{1}B_{1}^{\prime}}}
Fdy\\
&  =%
{\displaystyle\int\limits_{B_{1}^{\prime}Q_{1}B_{1}}}
\left(  -F\right)  dy\\
&  \geq%
{\displaystyle\int\limits_{H_{1}^{\prime}Q_{1}H_{1}}}
\left(  -F\right)  dy
\end{align*}
since, $F\left(  x\right)  <0$ and $dy>0$ for points on $H_{1}^{\prime}%
Q_{1}H_{1}$. Again since,
\[
\left.  F\left(  x\right)  \right]  _{B^{\prime}QB}\geq\left.  F\left(
x\right)  \right]  _{H_{1}^{\prime}Q_{1}H_{1}}%
\]
for same value of $y$ we get%

\begin{align}
V_{B_{1}Q_{1}B_{1}^{\prime}}  &  \geq%
{\displaystyle\int\limits_{H_{1}^{\prime}Q_{1}H_{1}}}
\left(  -F\right)  dy\geq%
{\displaystyle\int\limits_{B^{\prime}QB}}
\left(  -F\right)  dy\nonumber\\
&  =%
{\displaystyle\int\limits_{BQB^{\prime}}}
Fdy\nonumber\\
&  =V_{BQB^{\prime}}\nonumber\\%
\begin{tabular}
[c]{l}%
$\implies V_{B_{1}Q_{1}B_{1}^{\prime}}$%
\end{tabular}
&  \geq V_{BQB^{\prime}}. \label{Potential Comparison5}%
\end{align}

Next, let $S$ be a point on the curve of $F\left(  x\right)  $, to the right
of $A_{2}$, and let $BQB^{\prime}$ be an arbitrary path, with $Q$ to the right
of $S$. The straight line $PNSP^{\prime}$ is parallel to the $y-$ axis. Then,%
\begin{align}
V_{BQB^{\prime}}  &  =%
{\displaystyle\int\limits_{BQB^{\prime}}}
F\left(  x\right)  dy\nonumber\\
&  =%
{\displaystyle\int\limits_{B^{\prime}QB}}
\left(  -F\left(  x\right)  \right)  dy\nonumber\\
&  \geq%
{\displaystyle\int\limits_{P^{\prime}QP}}
\left(  -F\left(  x\right)  \right)  dy \label{Potential Comparison5 by y}%
\end{align}
since, $\left(  -F\left(  x\right)  \right)  \geq0$ and $dy\geq0$ along
$B^{\prime}QB$. Now by condition $\left(  iv\right)  $ of this theorem it
follows that $F$ is monotonic decreasing for $x>a_{2}$ and so we have
$\left\vert F\left(  x\right)  \right\vert \geq NS$ on $P^{\prime}QP$ and
since further $F\left(  x\right)  \leq0$ on $P^{\prime}QP$ so this implies
$-F\left(  x\right)  \geq NS$ on $P^{\prime}QP$. Again, $PP^{\prime}\geq
NP^{\prime}$. Thus we get%
\[
V_{BQB^{\prime}}\geq%
{\displaystyle\int\limits_{P^{\prime}QP}}
NS~dy=NS%
{\displaystyle\int\limits_{P^{\prime}QP}}
dy=NS\cdot PP^{\prime}\geq NS\cdot NP^{\prime}~\text{.}%
\]
But as $Q$ goes to infinity towards the right , $NP^{\prime}\rightarrow\infty
$. Hence, we can conclude that $V_{BQB^{\prime}}$ is positive and monotonic
increasing and tends to $+\infty$ as the paths recede to infinity.

\subsubsection*{Step $\left(  D\right)  :$}

From steps $\left(  A\right)  $ and $\left(  B\right)  $ it follows that the
quantities $V_{YX}+V_{X^{\prime}Y^{\prime}}$ and $V_{XB}+V_{B^{\prime
}X^{\prime}}$ are bounded quantities. Thus by $\left(
\ref{Potential Decomposition}\right)  $ and by step $\left(  C\right)  $ it
follows that $V_{YQY^{\prime}}$ is monotonic increasing to $+\infty$ to the
right of $A_{2}$.

\subsubsection*{Step $\left(  E\right)  :$}

By the construction of $\bar{\alpha}$ it is clear that $V_{YQY^{\prime}}<0$ in
$\bar{\alpha}\leq x<a_{2}$ i.e., to the left of $A_{2}$. Again from step
$\left(  D\right)  $ we conclude that $V_{YQY^{\prime}}$ ultimately becomes
positive as $Q$ moves out of $A_{2}$ along the curve of $F\left(  x\right)  $.
Therefore, by the same reason given in conclusion of the Theorem
$\ref{Lienard Theorem}$, it follows that there is one and only one path in the
region $x>\bar{\alpha}$ such that%
\[
V_{YQY^{\prime}}=0\text{.}%
\]
Also, by $\left(  \ref{Potential Zero}\right)  $ and the symmetry of the path
it is clear that the path is closed.

\subsubsection*{Step $\left(  F\right)  :$}

By the construction of $\bar{\alpha}$ and by step $\left(  E\right)  $ it is
clear that equation $\left(  \ref{Lienard Eq}\right)  $ has exactly two limit
cycles around the origin, the second limit cycle surrounds the first one. This
completes the proof of the Theorem $\ref{New Theorem}$.

\begin{remark}
It also follows from the proof that both the limit cycles are \textit{simple}
$($c.f., Remark $\ref{Simple Limit Cycle})$ that neither can bifurcate under
any small $C^{1}$ perturbation satisfying the conditions of the theorem
\end{remark}

\begin{remark}
One cannot assume that $V_{YQY^{\prime}}<0$ if $\bar{\alpha}\geq L$. We give a
counter example below.
\end{remark}

\begin{remark}
It is well known that two consecutive limit cycle cannot both be stable
$\left(  \text{unstable}\right)  $. Because of our choice of the function
$F\left(  x\right)  $ $($negative and monotone decreasing at the right of and
near the origin and infinity$)$, the inner limit cycle is stable and outer
limit cycle is unstable $($in reverse to those of reference $\left[
\cite{Odani N}\text{, }\cite{Holst Sunburg}\right]  )$.
\end{remark}

The existence of exactly $N$ limit cycles is established by an easy extension
of the above proof $\cite{Palit Datta 2}$. We state the theorem as follows. A
brief outline of its proof is given in the Appendix.

\begin{theorem}
\label{New Theorem N}Let $f$ and $g$ be two functions satisfying the following
properties.\newline%
\begin{tabular}
[c]{rl}%
$\left(  i\right)  $ & $f$ and $g$ are continuous;\\
$\left(  ii\right)  $ & $F$ and $g$ are odd functions and $g\left(  x\right)
>0$ for $x>0$.;\\
$\left(  iii\right)  $ & $F$ has $N$ number of positive simple zeros only at
$x=a_{i}$, $i=1,2,\ldots,N$\\
& where $0<a_{1}<a_{2}<\ldots<a_{N}$ such that in each interval $I_{i}=\left[
a_{i},a_{i+1}\right]  $,\\
& $i=1,2,\ldots,N-1$, there exists $\bar{\alpha}_{i}$, satisfying properties
given by $\left(  \text{\ref{Definition of alfaBar}}\right)  $,\\
& such that $\bar{\alpha}_{i}<L_{i}$ where $L_{i}$ is the unique extremum in
$I_{i}$,\\
& $i=1,\ldots,N-2$ and $L_{N-1}$, the first local extremum in $\left[
a_{N-1},a_{N}\right]  $.\\
$\left(  iv\right)  $ & $F$ is monotonic in $a_{i}<x\leq\bar{\alpha}_{i}$
$\forall$ $i$ and $\left\vert F\left(  x\right)  \right\vert \rightarrow
\infty$ as $x\rightarrow\infty$\\
& monotonically for $x>a_{N}$.
\end{tabular}
\newline Then the equation $\left(  \text{\ref{Lienard Eq}}\right)  $ has
exactly $N$ limit cycles around the origin, all are simple.
\end{theorem}

\begin{remark}
The conditions $\left(  i\right)  $ and $\left(  iv\right)  $ of Theorem
\ref{New Theorem} and Theorem \ref{New Theorem N} may be weakened following
Theorem \ref{Lienard Theorem Extension}. For instance, the condition $\left(
iv\right)  $ of Theorem \ref{New Theorem} may be restated as\newline%
\begin{tabular}
[c]{rl}%
$\left(  iv\right)  $ & $F$ is monotonic increasing in $a_{1}<x\leq\bar
{\alpha}=\bar{\alpha}_{1}$ $\left(  \text{say}\right)  $ and $\exists$ a
number $\bar{\alpha}_{2}$\\
& given by $\left(  \ref{Definition of alfaBar}\right)  $ such that $F$ is
monotonic decreasing in $a_{2}<x\leq\bar{\alpha}_{2}$ and\\
& nonincreasing in $\bar{\alpha}_{2}<x<d$.
\end{tabular}

\end{remark}

\section{Examples\label{Example}}

It is shown in section 3.3 of \cite{Holst Sunburg} that the limit cycles of
the autonomous system%
\begin{equation}
\left.
\begin{array}
[c]{c}%
\dot{x}=y\\
\dot{y}=-x-\mu h\left(  x,\dot{x}\right)
\end{array}
\right\}  \label{Canonical System}%
\end{equation}
are asymptotic to the circle $x^{2}+y^{2}=r^{2}$ as $\mu\rightarrow0$ where
the values of $r$ are the roots of the equation%
\begin{equation}
\Phi\left(  r\right)  :=\int_{0}^{2\pi}h\left(  r\sin u,r\cos u\right)  \cos
u~du=0\text{.} \label{Definition of phi}%
\end{equation}
We note that this is not the Lienard system. It is the canonical phase plane
for Lienard equation. The phase diagram of above system and the Lienard
system, however, should be similar. Here we take $\mu=0.1,$ $h\left(
x,\dot{x}\right)  =\left(  -4+75x^{2}-50kx^{4}\right)  \dot{x},$ $k\neq0$ so
that%
\[
f\left(  x\right)  =\mu\left(  -4+75x^{2}-50kx^{4}\right)  =-0.4+7.5x^{2}%
-5kx^{4}%
\]%
\[
F\left(  x\right)  =-0.4x+2.5x^{3}-kx^{5}%
\]
and%
\begin{align*}
\Phi\left(  r\right)   &  =\int_{0}^{2\pi}\left(  -4+75r^{2}\sin^{2}%
u-50kr^{4}\sin^{4}u\right)  r\cos u\cdot\cos u~du\\
&  =-\frac{1}{4}\pi r\left(  25kr^{4}-75r^{2}+16\right)  \text{.}%
\end{align*}
Therefore, $\left(  \ref{Definition of phi}\right)  $ reduces to%
\[
-\frac{1}{4}\pi r\left(  25kr^{4}-75r^{2}+16\right)  =0
\]
giving%
\[
r^{2}=\frac{1}{10k}\left(  15\pm\sqrt{225-64k}\right)  \text{.}%
\]
So, $\left(  \ref{Definition of phi}\right)  $ has real and distinct roots if
$225>64k$ i.e., if $k<3.515625$, real and repeated if $k=3.515625$ and
imaginary if $k>3.515625$. Therefore it follows that we will get two distinct
limit cycles, which are asymptotic to the circles corresponding to the above
two distinct values of $r$ if $k<3.515625$. Similarly, we will get only one
limit cycle when $k=3.515625$ and no limit cycle when $k>3.515625$. It can be
verified that the system undergoes a saddle node bifurcation at $k=3.515625$.

We note that the point $\left(  \bar{\alpha},F\left(  \bar{\alpha}\right)
\right)  $ on the limit cycle in the Lienard plane gets transformed to the
point $\left(  -\bar{\alpha},0\right)  $ lying on the almost circular limit
cycle of radius $r\gtrsim r_{1}$ $($in the canonical phase plane$)$ where
$r_{1}^{2}=\frac{1}{10k}\left(  15-\sqrt{225-64k}\right)  $ $(r_{1}$
corresponds to the first limit cycle$)$ under the transformation%
\[
x=-u,~y=-v+F\left(  u\right)
\]
$\left(  u,v\right)  $ and $\left(  x,y\right)  $ being the corresponding
points in Lienard plane and canonical phase plane with $f$ an even function.
We thus have%
\[
\bar{\alpha}\gtrsim\sqrt{\frac{1}{10k}\left(  15-\sqrt{225-64k}\right)
\text{.}}%
\]
We now present the phase diagrams of the above systems in Lienard plane in the
following examples for different values of $k$. These examples justify our new
theorem. We use Mathematica 5.1 in constructing the examples.

\begin{example}
\label{Ex1}Here we consider the autonomous system $\left(
\ref{Lienard System}\right)  $ with $k=3.65$, $f\left(  x\right)
=-0.4+7.5x^{2}-5kx^{4}$ $g\left(  x\right)  =x$ and $F\left(  x\right)
=-0.4x+2.5x^{3}-kx^{5}$. The phase diagram in Lienard plane is shown in Figure
$\ref{Ex Phase 1}$ \begin{figure}[ph]
\begin{center}
\mbox{
\subfigure[]{\includegraphics[height=6cm]{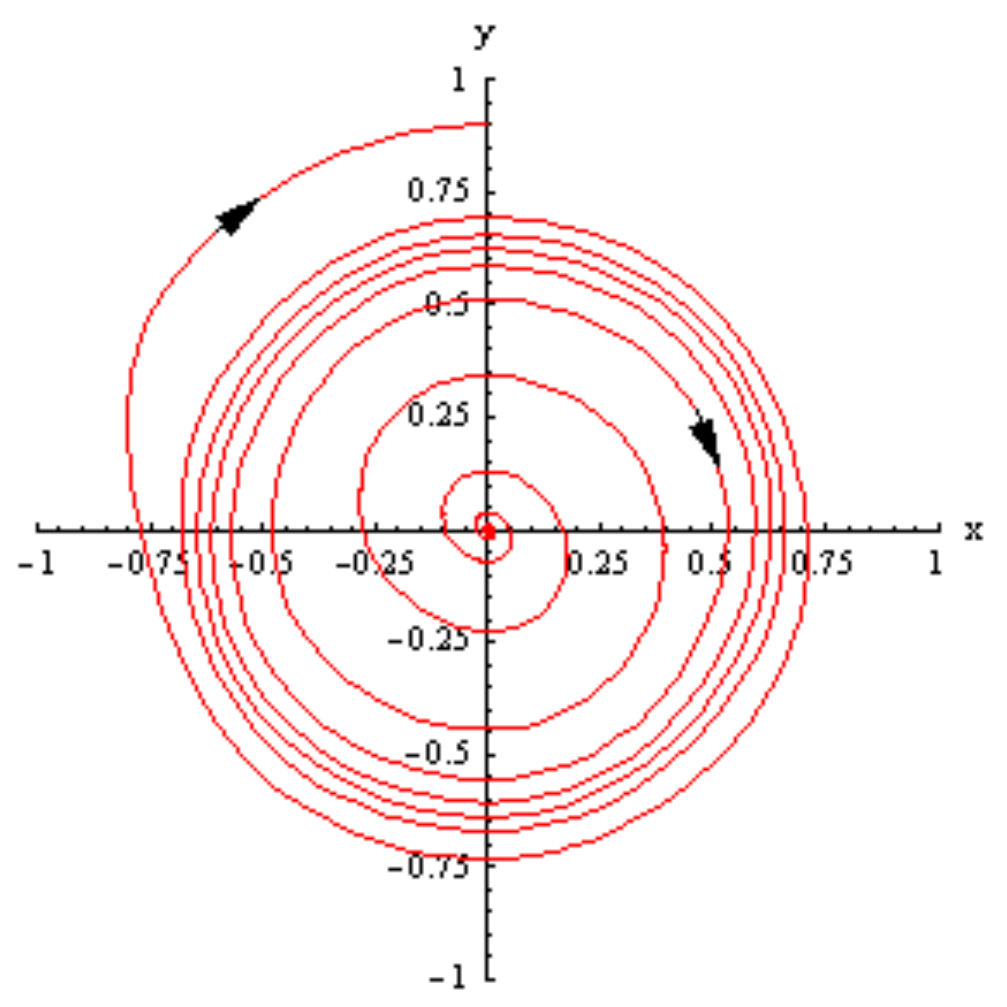}\label{Ex Phase 1}}
\subfigure[]{\includegraphics[height=6cm]{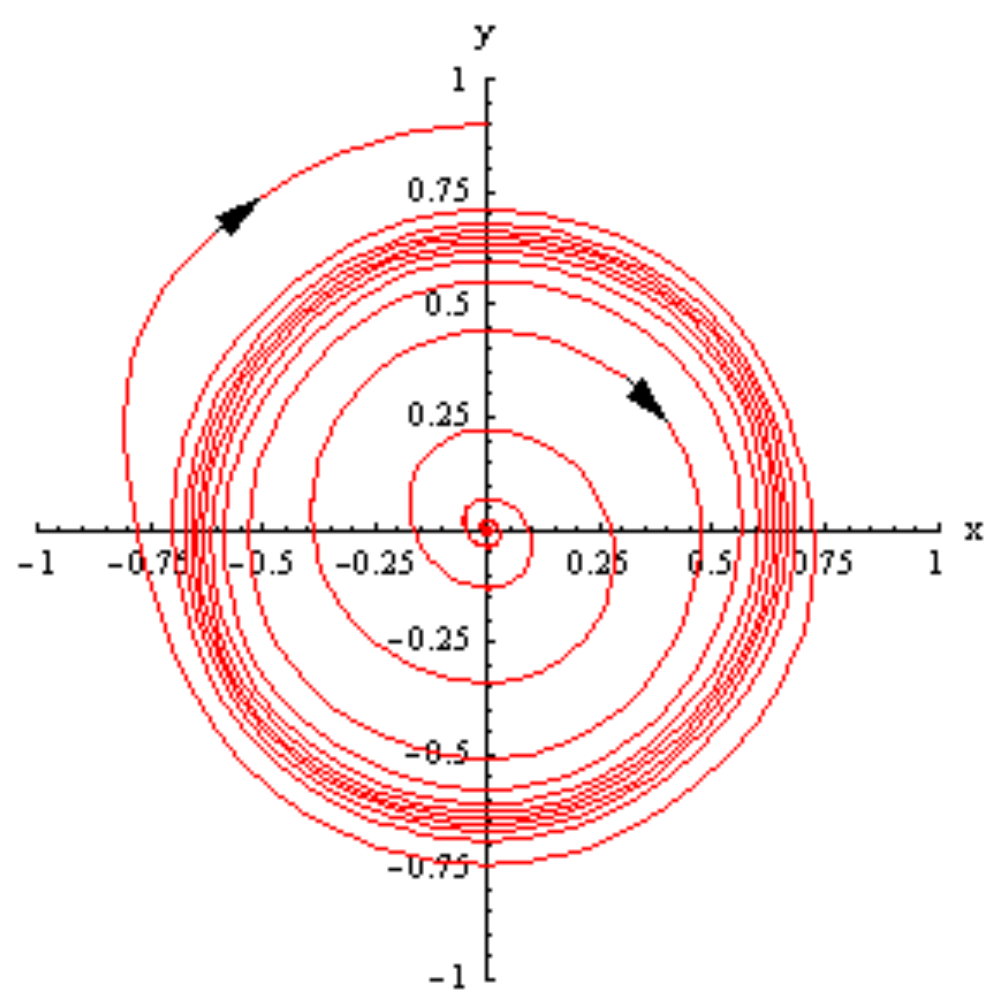}\label{Ex Phase 2}}
} \mbox{
\subfigure[]{\includegraphics[height=6cm]{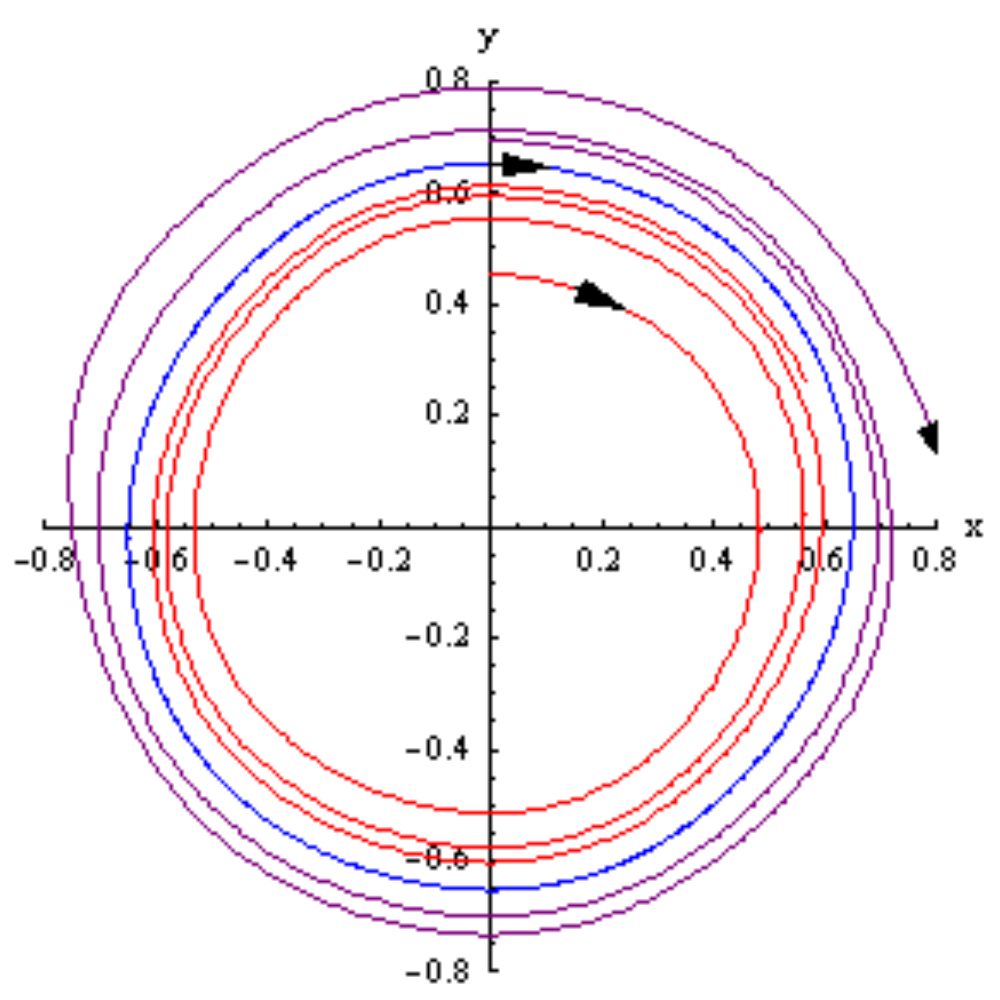}\label{Ex Unique Cycle}}
\subfigure[]{\includegraphics[height=6cm]{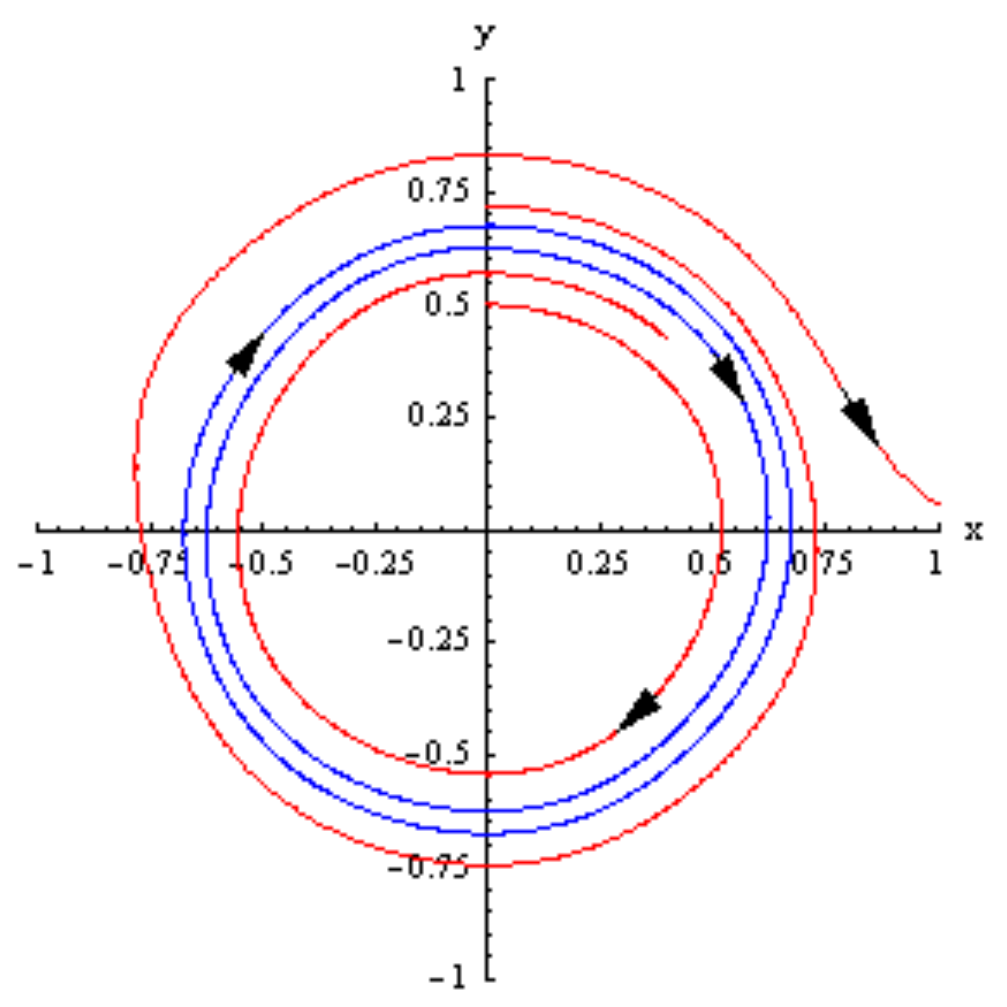}\label{Ex Not Suff 2 Cycle}}
}
\end{center}
\caption{{}The phase diagram of the system $\left(  \ref{Lienard System}%
\right)  $ in Lienard plane\newline$\left(  a\right)  $ with $k=3.65$, and
center as a repelling node in Example $\ref{Ex1}$,\newline$\left(  b\right)  $
with $k=3.57$ in Example $\ref{Ex1}$,\newline$\left(  c\right)  $ with
$k=3.515625$, and one limit cycle in Example $\ref{Ex One Limit Cycle}%
$,\newline$\left(  d\right)  $ with $k=3.5$, and two limit cycles in Example
$\ref{Ex Two Limit Cycle}$.}%
\label{Ex Phase Pic}%
\end{figure}which does not have any limit cycle. Again we take $k=3.57$ in the
above system. The corresponding phase diagram is shown in Figure
$\ref{Ex Phase 2}$. This phase diagram also does not contain any limit cycle,
but we see that the path is concentrating in a certain circular region.
\end{example}

\begin{example}
\label{Ex One Limit Cycle}Here we consider the autonomous system $\left(
\ref{Lienard System}\right)  $ discussed above with $k=3.515625$ so that
$a_{1}\simeq0.49307$, $a_{2}\simeq0.68410$, $y_{+}\left(  0\right)  $ and
$y_{-}\left(  0\right)  $ both are approximately equal to $0.652287$ and
$\bar{\alpha}\simeq0.65204$. Let $L_{1}$ be the point of minima of $F$ in
$\left(  0,a_{1}\right)  $ and $L_{2}$ be the point of maxima of $F$ in
$\left(  a_{1},a_{2}\right)  $. Here $F\left(  x\right)  $ is increasing in%
\[
\left(  -L_{2},L_{1}\right)  \cup\left(  L_{1},L_{2}\right)
\]
and decreasing in%
\[
\left(  -\infty,-L_{2}\right)  \cup\left(  -L_{1},L_{1}\right)  \cup\left(
L_{2},\infty\right)
\]
where, $L_{1}\simeq0.24997$ and $L_{2}\simeq0.60348$. In this case we obtain
only one limit cycle as shown in Figure $\ref{Ex Unique Cycle}$. Here, $F$ is
not monotone increasing throughout the interval $a_{1}<x\leq\bar{\alpha}$,
violating the condition $\left(  iv\right)  $ of Theorem $\ref{New Theorem}$
$($since $\bar{\alpha}>L_{2})$.
\end{example}

\begin{example}
\label{Ex Two Limit Cycle}We now take $k=3.5$. Here, $y_{+}\left(  0\right)  $
and $y_{-}\left(  0\right)  $ are approximately equal to $0.624499$. The
equations $\left(  \ref{Definition of alfa+}\right)  $ and $\left(
\ref{Definition of alfa-}\right)  $ both reduce to%
\[
\frac{x^{2}}{2}=\frac{1}{2}\left(  0.624499\right)  ^{2}-\frac{1}{2}\left(
-0.4x+2.5x^{3}-3.5x^{5}\right)  ^{2}%
\]
having real roots $x=\pm0.62393$ so that $\bar{\alpha}\simeq0.62393$. Here,
$a_{1}=0.4919$ and $a_{2}=0.68725$ showing that $a_{1}<\bar{\alpha}<a_{2}$.
Next, $L_{1}\simeq0.24985$, $L_{2}\simeq0.60510$. Here all the conditions of
Theorem $\ref{New Theorem}$ are satisfied except condition $\left(  iv\right)
$. However we still get two limit cycles as shown in Figure
$\ref{Ex Not Suff 2 Cycle}$ drawn in Lienard plane. This example and the above
example show that the conditions of Theorem $\ref{New Theorem}$ are sufficient
but not necessary.
\end{example}

\begin{example}
\label{Ex4}Finally we take $k=3$. Here, $y_{+}\left(  0\right)  $ and
$y_{-}\left(  0\right)  $ are approximately equal to $0.5552$ and $\bar
{\alpha}\simeq0.55324$. Here, $a_{1}=0.46473$ and $a_{2}=0.78572$ showing that
$a_{1}<\bar{\alpha}<a_{2}$. Next, $L_{1}\simeq0.24638$, $L_{2}\simeq0.66279$.
Here, all the conditions of Theorem $\ref{New Theorem}$ are satisfied and so
we get exactly two limit cycles. The phase diagram in Lienard plane is shown
in Figure $\ref{Ex 2 Cycle}$.\begin{figure}[th]
\begin{center}
\includegraphics[height=6cm]{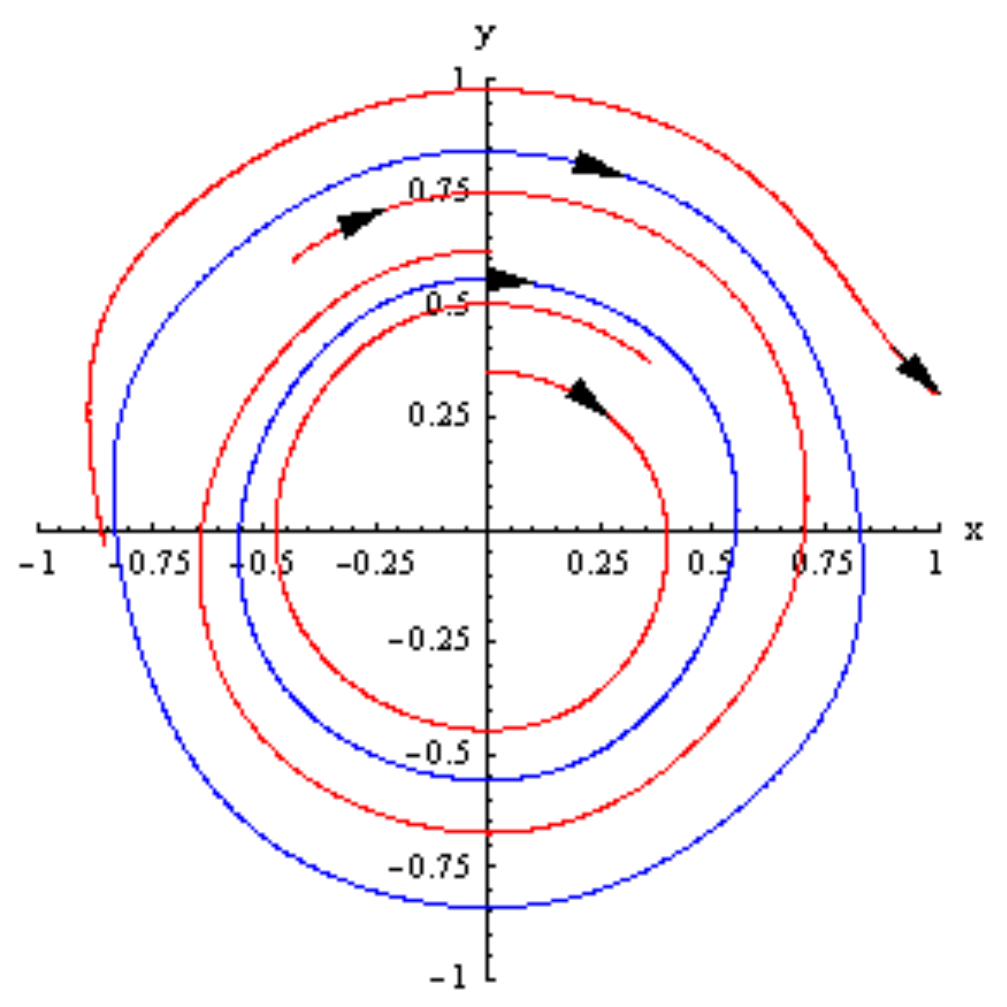}
\end{center}
\caption{{}The phase diagram of the system $\left(  \text{\ref{Lienard System}%
}\right)  $ in Lienard plane with $k=3$, and two limit cycles in Example
$\ref{Ex4}$.}%
\label{Ex 2 Cycle}%
\end{figure}
\end{example}

\begin{remark}
Although in the above examples the value of $\mu$ is sufficiently small $($so
as to satisfy the amplitude analysis of $\cite{Holst Sunburg})$ our theorem
should be applicable for large values of $\left\vert \mu\right\vert $. More
detailed bifurcation analysis in the $\left(  \mu,k\right)  $ parametric plane
will be considered separately.
\end{remark}

\begin{example}
\label{Ex Compare 2 Cycle}We now consider the function%
\[
F_{+}\left(  x\right)  =\left\{
\begin{array}
[c]{ll}%
-0.1\sin\left(  10\pi x\right)  & 0\leq x<0.15\\
0.01\sqrt{1-\left(  \dfrac{x-0.15}{0.01}\right)  ^{2}} & 0.15\leq
x<0.15+\dfrac{1}{\sqrt{101}}\\%
\begin{array}
[c]{l}%
0.02099503719021-\\
\hspace{0.5in}2\sqrt{0.1\left(  x-0.2395037190209989\right)  }%
\end{array}
& x\geq0.15+\dfrac{1}{\sqrt{101}}%
\end{array}
\right.
\]
Then we have $a_{1}=0.1$, $a_{2}=0.25052350868645645$, $L=0.15$. Here,
\[
f\left(  L\right)  =0<f\left(  0.2395037190209989\right)
=0.2006848039831627>f\left(  a_{2}\right)  =0.9526060763219791,
\]
though
\[
L<0.2395037190209989<a_{2}%
\]
showing that the function is not monotone nonincreasing in $\left[
L,a_{2}\right]  $ and so it does not satisfy the condition $\left(  3\right)
$ of Theorem $5.1$ in chapter $4$ in the book $\cite{Zhing Tongren Wenzao}$.
However for the inner limit cycle we have $y_{+}\left(  0\right)
=y_{-}\left(  0\right)  =0.12238318$. So, $\bar{\alpha}=0.12221435874426823<L$
satisfying the conditions of Theorem $\ref{New Theorem}$. This example clearly
shows that the Theorem $\ref{New Theorem}$ covers a larger class of functions
than those covered by Theorem $5.1$ in chapter $4$ in $\cite{Zhing Tongren
Wenzao}$. The function $F$ alongwith two limit cycles are shown in Figure
$\ref{Ex Compare 2 Cycle Pic}$.\begin{figure}[h]
\begin{center}
\includegraphics[height=6cm]{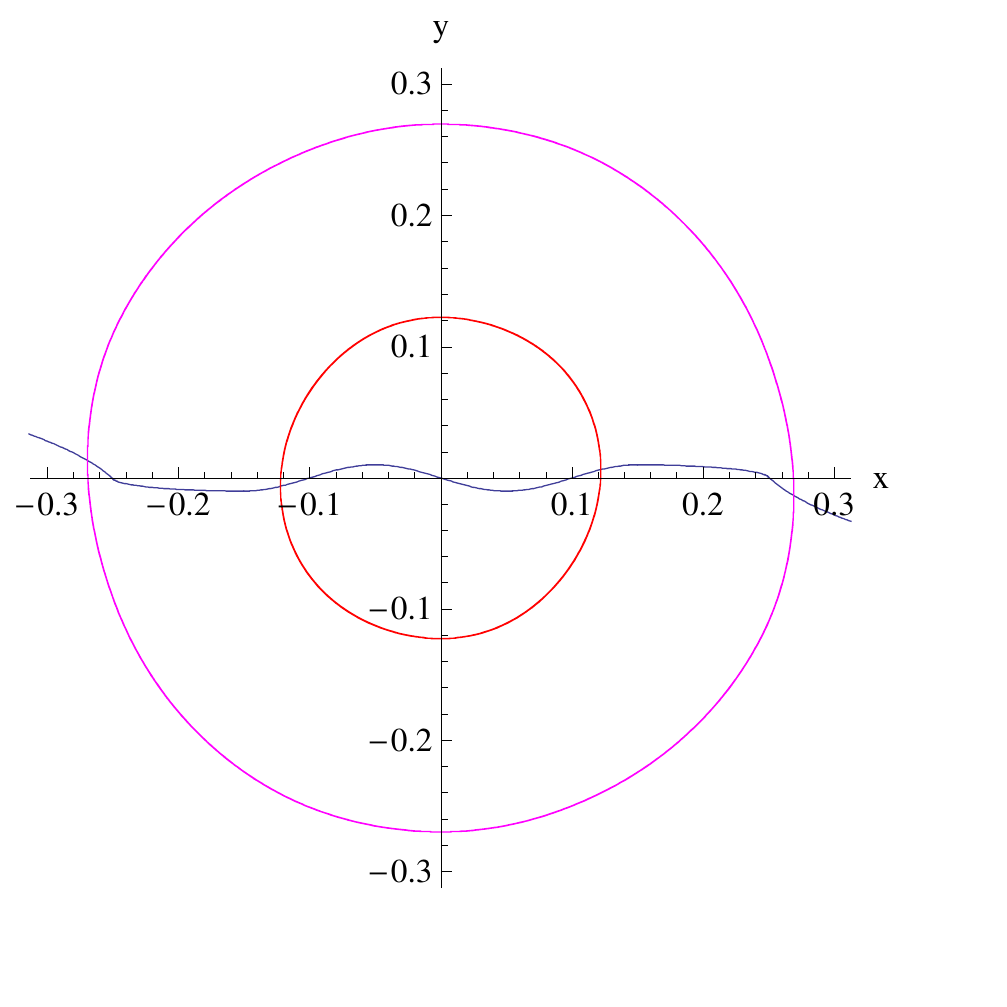}
\end{center}
\caption{{}The phase diagram of $\left(  \text{\ref{Lienard System}}\right)  $
in Lienard plane for Example \ref{Ex Compare 2 Cycle}.}%
\label{Ex Compare 2 Cycle Pic}%
\end{figure}
\end{example}

\begin{example}
\label{Ex 3 Limit Cycle}We now consider a different problem. Here we define%
\[
F_{+}\left(  x\right)  =\left\{
\begin{array}
[c]{ll}%
0.005-0.025\sqrt{1-\left(  \dfrac{x-0.048989794}{0.05}\right)  ^{2}} & 0\leq
x<a_{1}\\
-0.0008137888130718+0.01\sqrt{1-\left(  \dfrac{x-0.14781375}{0.05}\right)
^{2}} & a_{1}\leq x<a_{2}\\
0.0009168416064002765-0.015\sqrt{1-\left(  \dfrac{x-0.29746094}{0.1}\right)
^{2}} & a_{2}\leq x<a_{3}\\
-0.0003265987749816556+0.04\sqrt{x-0.3972073012751128} & x\geq a_{3}%
\end{array}
\right.
\]
where%
\begin{align*}
a_{1}  &  =0.097979588\\
a_{2}  &  =0.197647912\\
\text{and }a_{3}  &  =0.397273968.
\end{align*}
and%
\[
F\left(  x\right)  =\left\{
\begin{array}
[c]{lc}%
F_{+}\left(  x\right)  & x\geq0\\
-F_{+}\left(  -x\right)  & x<0
\end{array}
\right.
\]
The function $F_{+}\left(  x\right)  $ is obtained by matching three ellipses
and a parabola successively in the intervals $\left(  0,a_{1}\right)  ,$
$\left(  a_{1},a_{2}\right)  ,$ $\left(  a_{2},a_{3}\right)  ,$ and $\left(
a_{3},\infty\right)  $ such that%
\begin{equation}
F_{+}\left(  a_{i}+0\right)  =F_{+}\left(  a_{i}-0\right)  \text{ and }%
F_{+}^{\prime}\left(  a_{i}+0\right)  =F_{+}^{\prime}\left(  a_{i}-0\right)  ,
\label{Matching Cond}%
\end{equation}
where $a_{i}$'s are zeros of $F_{+}$. The unique extremum of $F$ in $\left(
0,a_{1}\right)  $, $\left(  a_{1},a_{2}\right)  $, $\left(  a_{2}%
,a_{3}\right)  $ are respectively%
\begin{align*}
L_{0}  &  =0.048989794\\
L_{1}  &  =0.14781375\\
\text{and }L_{2}  &  =0.29746094.
\end{align*}
We obtain three limit cycles which meet the positive $y$-axis at the points
$\left(  0,y_{1}\left(  0\right)  \right)  $, $\left(  0,y_{2}\left(
0\right)  \right)  $, $\left(  0,y_{3}\left(  0\right)  \right)  $ where%
\begin{align*}
y_{1}\left(  0\right)   &  =0.1332869\\
y_{2}\left(  0\right)   &  =0.212146685\\
\text{and }y_{3}\left(  0\right)   &  =0.4630114\text{.}%
\end{align*}
The matching conditions $\left(  \ref{Matching Cond}\right)  $ are used to
make $F\in C^{1}\left(  R\right)  $ with accuracy level $O\left(
10^{-7}\right)  $. This function is constructed in a trial and error method
and numerical data with large significant digits arise in this fashion.
Examples with lower significant digits and lower and higher accuracy are
possible in principle. Here we get $\bar{\alpha}_{1}=0.133002186$ and
$\bar{\alpha}_{2}=0.21203506657$. The function $F$ satisfies all the
conditions of Theorem $\ref{New Theorem N}$ $($for example $\bar{\alpha}%
_{i}<L_{i}$ etc.$)$ and so the existence of the above three limit cycles are
ensured by this theorem, the proof of which is presented separately
$\cite{Palit Datta 2}$. However, the function $F$ is defined in such a manner
that $\left\vert F\left(  L_{0}\right)  \right\vert >\left\vert F\left(
L_{2}\right)  \right\vert $ implying that $\beta_{2}$ mentioned in Theorem $1$
of $\cite{Chen Llibre Zhang}$ or in Theorem $7.12$, chapter $4$ of the book
$\cite{Zhing Tongren Wenzao}$ does not exist and hence these theorems are not
applicable for the corresponding Lienard system. The limit cycles of the
Lienard system in Lienard plane and the graph of the function $F$ have been
shown separately in the Figure $\ref{Ex 3 Cycle}$. To conclude, Theorem $1$ of
$\cite{Chen Llibre Zhang}$ or Theorem $7.12$, Chapter $4$ in the book
$\cite{Zhing Tongren Wenzao}$ fail to predict the existence of the exact
number of limit cycles for the above function $F\left(  x\right)
$.\begin{figure}[th]
\begin{center}
\mbox{
\subfigure[]{\includegraphics[height=6cm]{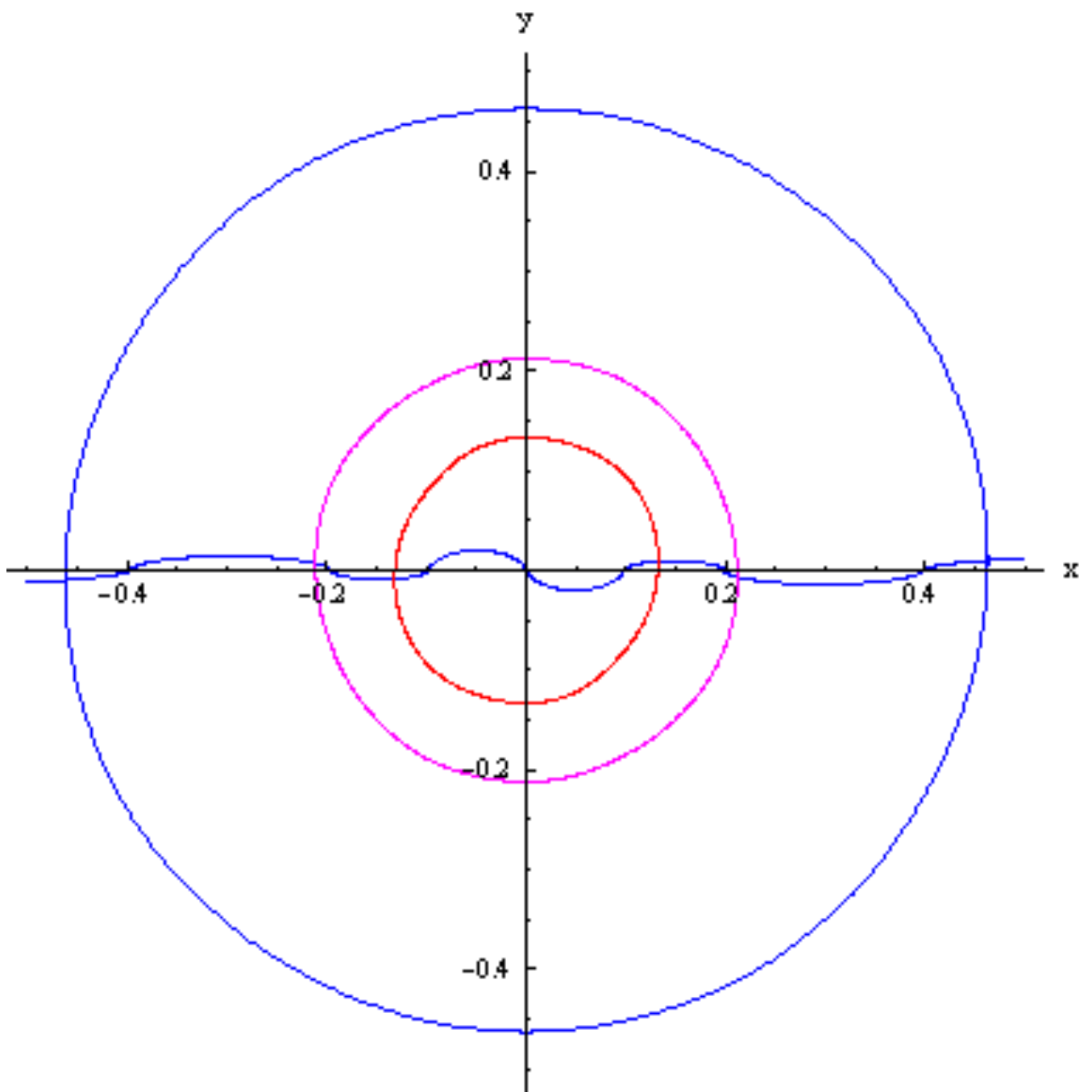}\label{Ex 3 Cycle a}}
\subfigure[]{\includegraphics[height=6cm]{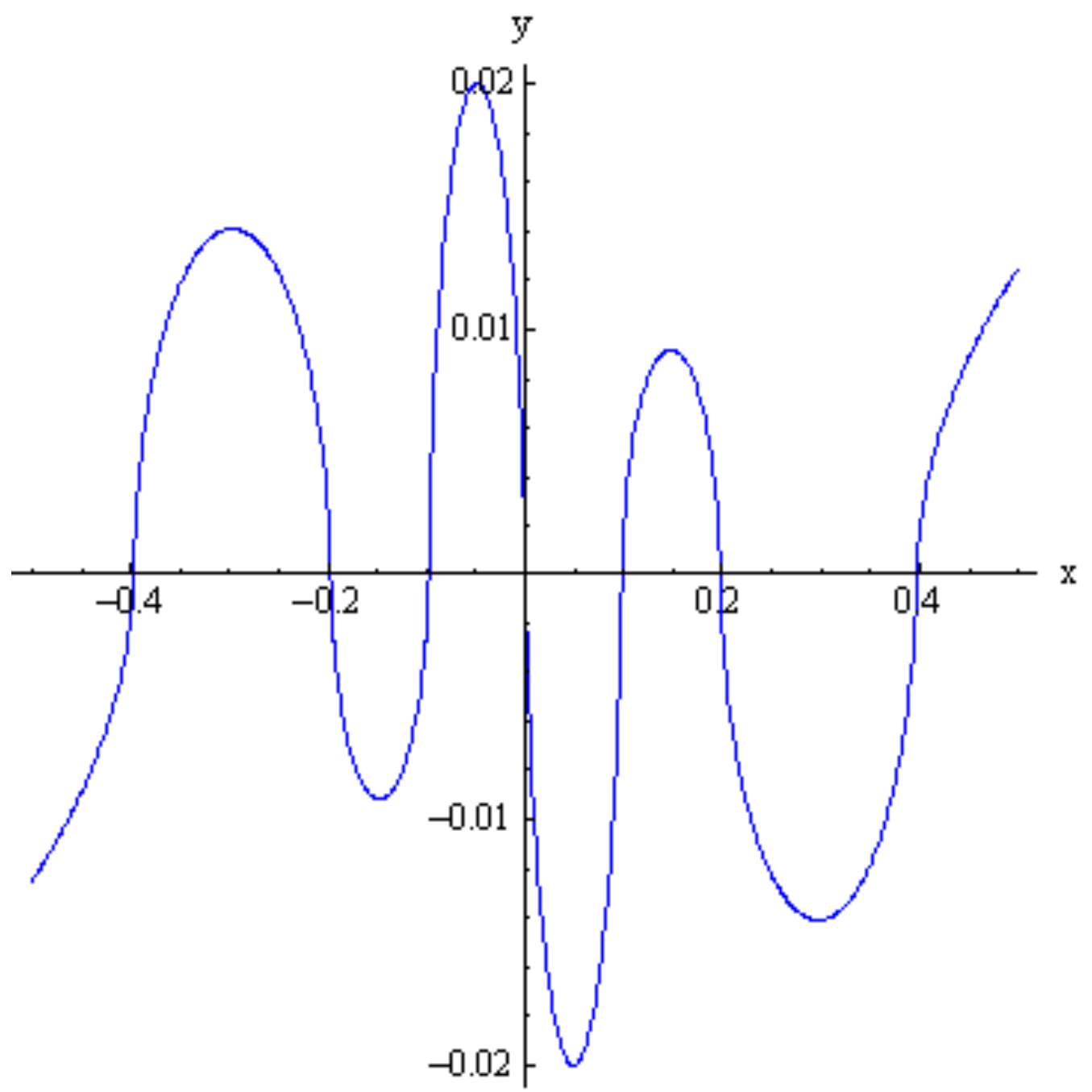}\label{Ex 3 Cycle b}}
}
\end{center}
\caption{$\left.  {}\right.  \newline\left(  a\right)  $ The phase diagram of
the system $\left(  \ref{Lienard System}\right)  $ in Lienard plane with three
limit cycles.\newline$\left(  b\right)  $ Graph of the function $F$ in Example
$\ref{Ex 3 Limit Cycle}$.}%
\label{Ex 3 Cycle}%
\end{figure}
\end{example}

\section{Concluding Remarks\label{Current Status}}

Many interesting new results have been proved on the existence of an exact
number of multiple limit cycles $\left[  \cite{Odani N},\cite{Holst
Sunburg},\cite{Chen Llibre Zhang}\right]  $ in the recent past. Odani has
proved a sufficient condition in $\cite{Odani N}$ using a choice function
$\phi_{k}$ which can be exploited to obtain better estimates of amplitudes of
the limit cycles. We used a straight forward method depending on the geometry
of phase diagram. We have proved a similar result with more general class of
functions $F\left(  x\right)  $ by a simpler method. In the present approach a
strict monotonicity of $F\left(  x\right)  $ is required only in the intervals
$a_{1}<x<\bar{\alpha}$ and $x>a_{2}$. Consequently, $F\left(  x\right)  $ can
accommodate \textquotedblleft small scale" oscillations in the interval
$\bar{\alpha}<x<a_{2}$. Odani, for instance, considered an $F\left(  x\right)
$ which is not only $C^{1}$ but also has a unique extremum in the interval
$a_{1}<x<a_{2}$. Further, the theorem is valid for a more general class of the
function $g\left(  x\right)  $. Odani's theorem however is valid only for
$g\left(  x\right)  =x$. An interesting problem will be to establish the
relation between $\bar{\alpha}$ of our approach and the function $\phi_{k}$.
We note that $\hat{\alpha}_{i}$ corresponds to the amplitude of the limit
cycles. Our estimates of amplitude of the limit cycle of the Van der pol
equation constitute an improvement over those available in the literature
$\left[  \cite{Odani},\text{ }\cite{Lopez}\right]  $. Examples
$\ref{Ex Compare 2 Cycle}$ \& $\ref{Ex 3 Limit Cycle}$, on the other hand,
show the difference between the present theorem and those of $\cite{Zhing
Tongren Wenzao}$ and $\cite{Chen Llibre Zhang}$. The calculations are accurate
upto the accuracy level $O\left(  10^{-7}\right)  $. The existence of limit
cycles in a Lienard system allowing discontinuity $($see, for instance,
$\cite{Llibre Ponce Torres})$ is an interesting problem for further study.
Determining the shape of the limit cycles is also left for future investigations.

Before closing we note that the value $\bar{\alpha}$, in general, is a
function of the parameters of $F\left(  x\right)  $ in the parametric space.
For instance, in Examples $\ref{Ex1}-\ref{Ex4}$, $\bar{\alpha}$ is a function
of the parameters $\mu$ and $k$. The study of the variation of $\bar{\alpha}$
in the parametric space seems to offer interesting insights into the
bifurcation and related issues of the multiple limit cycles in a Lienard
system. The relationship with Poincare's return map also needs to be studied.
We wish to investigate these problems in future.

\section{Appendix}

Here we present a brief outline of the proof of Theorem $\ref{New Theorem N}$
which is given in detail in $\cite{Palit Datta 2}$. The theorem is proved by
an induction method which is dependent on the \textit{non-trivial} initial
hypotheses that the the result holds for $N=1$ and $N=2$.

\begin{figure}[h]
\begin{center}
\includegraphics[height=8cm,width=10cm]{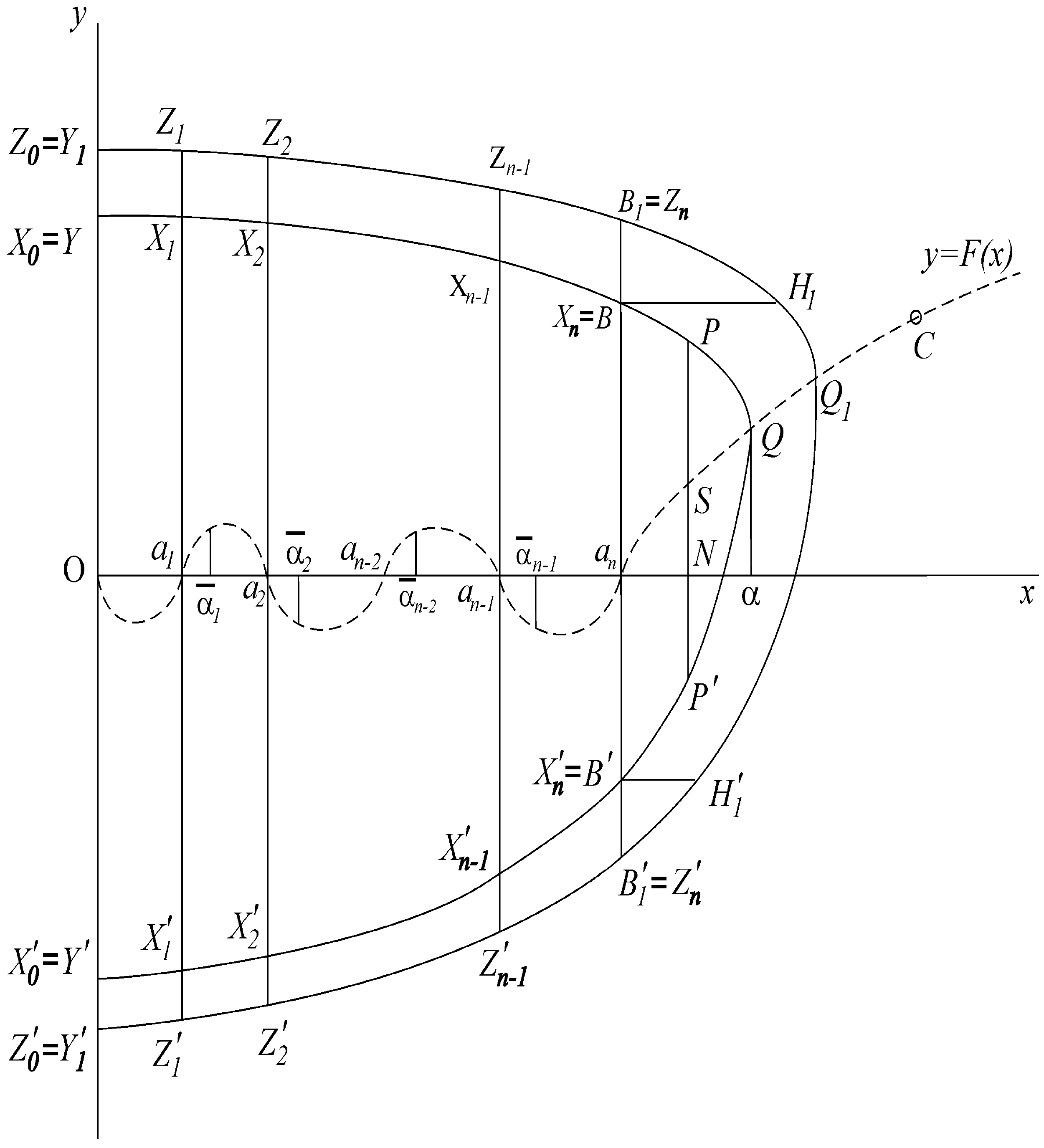}
\end{center}
\caption{{}}%
\label{Typical n Path}%
\end{figure}We shall prove the theorem by showing the result that each limit
cycle intersects the $x-$axis at a point lying in the open interval $\left(
\bar{\alpha}_{i},\bar{\alpha}_{i+1}\right]  $, $i=0,1,2,\ldots,N-1$, where
$\bar{\alpha}_{0}=L_{0}$ is the local minima of $F\left(  x\right)  $ in
$\left[  0,a_{1}\right]  $. By Lienard theorem and Theorem $\ref{New Theorem}$
it follows that the result is true for $N=1$ and $N=2$. We shall now prove the
theorem by induction. We assume that the theorem is true for $N=n-1$ and we
shall prove that it is true for $N=n$. We prove the theorem by taking $n$ as
an odd $+ve$ integer so that $\left(  n-1\right)  $ is even. The case for
which $n$ is even can similarly be proved and so is omitted. It can be shown
that $\cite{Jordan Smith}$, $V_{YQY^{\prime}}$ changes its sign from $+ve$ to
$-ve$ as $Q$ moves out of $A_{1}\left(  a_{1},0\right)  $ along the curve
$y=F\left(  x\right)  $ and hence vanishes there due to its continuity and
generates the first limit cycle around the origin. Next, in Theorem
$\ref{New Theorem}$ we see $V_{YQY^{\prime}}$ again changes its sign from
$-ve$ to $+ve$ and generates the second limit cycle around the first. Also, we
see that for existence of second limit cycle we need the existence of the
point $\bar{\alpha}$, which we denote here as $\bar{\alpha}_{1}$.

Since by induction hypothesis the theorem is true for $N=n-1$, so it follows
that in each and every interval $\left(  \bar{\alpha}_{k},\bar{\alpha}%
_{k+1}\right]  $, $k=0,1,2,\ldots,n-2$ the system $\left(
\ref{Lienard System}\right)  $ has a limit cycle and the outermost limit cycle
cuts the $x-$ axis somewhere in $\left(  \bar{\alpha}_{n-1},\infty\right)  $.
Also $V_{YQY^{\prime}}$ changes its sign alternately as the point $Q$ moves
out of $a_{i}$'s, $i=1,2,\ldots,n-1$. Since $\left(  n-1\right)  $ is even, it
follows that $V_{YQY^{\prime}}$ changes its sign from $+ve$ to $-ve$ as $Q$
moves out of $a_{n-2}$ along the curve $y=F\left(  x\right)  $. Since there is
only one limit cycle in the region $\left(  \bar{\alpha}_{n-1},\infty\right)
$, so it is clear that $V_{YQY^{\prime}}$ must change its sign from $-ve$ to
$+ve$ once and only once as $Q$ moves out of $A_{n-1}\left(  a_{n-1},0\right)
$ along the curve $y=F\left(  x\right)  $. Also it follows that once
$V_{YQY^{\prime}}$ becomes $+ve$, it can not vanish further, otherwise we
would get one more limit cycle, contradicting the hypothesis so that total
number of limit cycle becomes $n$. We now try to find an estimate of $\alpha$
for which $V_{YQY^{\prime}}$ vanishes for the last time.

We shall now prove that the result is true for $N=n$ and so we assume that all
the hypotheses or conditions of this theorem are true for $N=n$. So, we get
one more point $\bar{\alpha}_{n}$ and another root $a_{n}$, ensuring the fact
that $V_{YQY^{\prime}}$ vanishes as $Q$ moves out of $A_{n-1}$ through the
curve $y=F\left(  x\right)  $, thus accommodating a unique limit cycle in the
interval $\left(  \bar{\alpha}_{n-1},\bar{\alpha}_{n}\right]  $.

By the result discussed so far it follows that $V_{YQY^{\prime}}>0$ when
$\alpha$ lies in certain suitable small right neighbourhood of $\bar{\alpha
}_{n-1}$. We shall prove that $V_{YQY^{\prime}}$ ultimately becomes $-ve$ and
remains $-ve$ as $Q$ moves out of $A_{n}\left(  a_{n},0\right)  $ along the
curve $y=F\left(  x\right)  $ generating the unique limit cycle and hence
proving the required result for $N=n$.

We draw straight line segments $X_{k}X_{k}^{\prime}$, $k=1,2,3,\ldots,n$,
passing through $A_{k}$ and parallel to $y$-axis as shown in Figure
$\ref{Typical n Path}$. For convenience, we shall call the points $X_{n}%
,X_{n}^{\prime},Y,Y^{\prime}$ as $B,B^{\prime},X_{0},X_{0}^{\prime}$
respectively. We write the curves\newline$\left.  {}\right.  $\hfill
$\Gamma_{k}=X_{k-1}X_{k},\quad\Gamma_{k}^{\prime}=X_{k}^{\prime}%
X_{k-1}^{\prime},\quad k=1,2,3,\ldots,n$\hfill$\left.  {}\right.  $\newline so
that\newline$\left.  {}\right.  $\hfill$YQY^{\prime}=X_{0}QX_{0}^{\prime}=%
{\textstyle\sum\limits_{k=1}^{n}}
\Gamma_{k}+X_{n}QX_{n}^{\prime}+%
{\textstyle\sum\limits_{k=1}^{n}}
\Gamma_{k}^{\prime}=%
{\textstyle\sum\limits_{k=1}^{n}}
\left(  \Gamma_{k}+\Gamma_{k}^{\prime}\right)  +BQB^{\prime}$\hfill$\left.
{}\right.  $\newline and%
\begin{equation}
V_{YQY^{\prime}}=%
{\textstyle\sum\limits_{k=1}^{n}}
\left(  V_{\Gamma_{k}}+V_{\Gamma_{k}^{\prime}}\right)  +V_{BQB^{\prime}%
}\text{.} \label{Potential Sum}%
\end{equation}
which is used in place of the function in $\left(
\ref{Potential Decomposition}\right)  $. The rest of the proof are analogous
to that of Theorem $\ref{New Theorem}$ and proved separately in $\cite{Palit
Datta 2}$.

\end{document}